\documentclass{amsart}
\usepackage{latexsym}
\usepackage{amsmath}
\usepackage{amssymb}
\usepackage[all]{xy}

\newtheorem{thm}{Theorem}[section]
\newtheorem{prop}[thm]{Proposition}
\newtheorem{cor}[thm]{Corollary}
\newtheorem{lem}[thm]{Lemma}
\theoremstyle{definition}

\theoremstyle{remark}
\newtheorem{rem}[thm]{Remark}
\newtheorem{ex}[thm]{Example}

\newcommand{\aut}[1]{\text{\rm aut}_1({#1})}
\newcommand{\K}{{\mathbb Q}}
\newcommand{\F}{{\mathcal F}}
\newcommand{\e}{\varepsilon}
\newcommand{\mapright}[1]{%
 \smash{\mathop{%
  \hbox to 1cm{\rightarrowfill}}\limits_{#1}}}
\newcommand{\maprightd}[2]{%
 \smash{\mathop{%
  \hbox to 1.2cm{\rightarrowfill}}\limits^{#1}\limits_{#2}}}
\newcommand{\mapleft}[1]{%
 \smash{\mathop{%
  \hbox to 1cm{\leftarrowfill}}\limits_{#1}}}
\newcommand{\mapleftu}[1]{%
 \smash{\mathop{%
  \hbox to 0.8cm{\leftarrowfill}}\limits^{#1}}}
\newcommand{\maprightu}[1]{%
 \smash{\mathop{%
  \hbox to 1cm{\rightarrowfill}}\limits^{#1}}}
\newcommand{\maprightud}[2]{%
 \smash{\mathop{%
  \hbox to 1cm{\rightarrowfill}}\limits^{#1}_{#2}}}
\newcommand{\mapleftud}[2]{%
 \smash{\mathop{%
  \hbox to 1cm{\leftarrowfill}}\limits^{#1}_{#2}}}

\newcounter{eqn}[section]

\def\theeqn{\textnormal{(\thesection.\arabic{eqn})}}

\def\eqnlabel#1{%
  \refstepcounter{eqn}%
  \label{#1}%
  \leqno{\theeqn}}

\begin{document}

\title[Rational visibility of a Lie group ]
{Rational visibility of a Lie group in the monoid of self-homotopy
equivalences of a homogeneous space}


\footnote[0]{ 

This research was partially supported by a Grant-in-Aid for Scientific
Research (C)20540070
from Japan Society for the Promotion of Science.

Department of Mathematical Sciences, 
Faculty of Science,  
Shinshu University,   
Matsumoto, Nagano 390-8621, Japan   
e-mail:{\tt kuri@math.shinshu-u.ac.jp}

}

\author{Katsuhiko KURIBAYASHI}
\date{}
   
\maketitle

\begin{abstract}
Let $M$ be a homogeneous space admitting 
a left translation by a connected Lie group $G$. 
The adjoint to the action gives rise to a map from $G$ to the monoid of 
 self-homotopy equivalences of $M$.
The purpose of this paper is to investigate 
the injectivity of the homomorphism which is induced by the adjoint map 
on the rational homotopy. In particular, {\it the visible degrees} are 
determined explicitly for all the cases of simple Lie groups and their
 associated homogeneous spaces of rank one which are classified by
 Oniscik.
\end{abstract}

\section{Introduction}
The study of rational visibility problems we here consider is motivated 
by work due to Kedra and McDuff \cite{K-M} in which symplectic
topological methods are effectively used. 
In this paper, we deal with such problems relying upon algebraic models
for spaces and maps, which are
viewed as complements of those developed and used 
in recent work on rational homotopy of functions spaces 
\cite{B-L, B-F-M, B-M, B-M2, L-S, L-S2, L-S3}.

Let $f : X \to Y$ be a map between connected spaces 
whose fundamental groups are abelian.  
We say that $X$ is {\it rationally visible} in $Y$ 
with respect to the map $f$ if the induced map 
$
f_*\otimes 1 : \pi_i(X)\otimes \K \to  \pi_i(Y)\otimes \K 
$ 
is injective for any $i\geq 1$.  
Let $\aut{M}$ denote the identity component of 
the monoid of self-homotopy equivalences of a space $M$.  
Let $G$ be a connected Lie group and $M$  
an appropriate homogeneous space $M$ admitting a left translation by 
$G$.  
We then define a map of monoids 
$$
\lambda_{G, M} : G \to \aut{M} 
$$   
by $\lambda_{G, M}(g)(x)= gx$ for $g \in G$ and $x \in M$. 
In this paper, we investigate the rational visibility of $G$ 
in $\aut{M}$ with respect to the map $\lambda_{G, M}$.

The monoid map $\lambda_{G, M}$ factors through 
the identity component $\text{Homeo}_1(M)$ of 
the monoid of homeomorphisms of $M$ as well as the identity component 
$\text{Diff}_1(M)$ of the space of diffeomorphisms of $M$. 
Therefore the rational visibility of $G$ in $\aut{M}$ implies that 
of $G$ in $\text{Homeo}_1(M)$ and $\text{Diff}_1(M)$. We also expect
that non-trivial characteristic classes of the classifying spaces 
 $B\aut{M}$, $B\text{Homeo}_1(M)$ and $B\text{Diff}_1(M)$ can be obtained 
through the study of rational visibility.  
Very little is known about the (rational) homotopy of the groups   
$\text{Homeo}_1(M)$ and $\text{Diff}_1(M)$ for a general manifold $M$; 
see \cite{F-H} for the calculation of
$\pi_i(\text{Diff}_1(S^n))\otimes \K$ for $i$ in some range.  
Then such implication and expectation inspire us to consider 
the visibility problems of Lie groups. 
We refer the reader to papers \cite{F-T} and \cite{S.Smith2} 
for the study of rational homotopy types of $\text{aut}_1(M)$ itself and  
related function spaces.  

The key device for the study of rational visibility is the 
function space model due to Brown and Szczarba \cite{B-S} and 
Haefliger \cite{H}. 
Especially, an explicit rational model for the
map $\lambda_{G,M}$ is constructed by using that for 
the evaluation map described in \cite{B-M} and \cite{Ku1}; 
see Theorem \ref{thm:model_adjoint}.  
By analyzing such elaborate models, we obtain 
a recognition principle for rational visibility in Theorem 
\ref{thm:key} below. 
We here emphasize that not only does our machinery developed 
in this paper allow us to give other proofs 
to results in \cite{K-M}, \cite{N-S},  \cite{S} and  \cite{Y} 
concerning rational visibility 
but also it leads us to an unifying way of looking at  
the visibility problem explicitly as is seen in Tables 1 and 2 below.

In the rest of this section, we state our results.  

\begin{thm}
\label{thm:main1}
Let $G$ be a simply-connected Lie group and $T$ a torus in
 $G$ which is not necessarily maximal. Then $G$ is rationally visible in 
$\text{\em aut}_1(G/T)$ with respect to the map $\lambda_{G, G/T}$
 defined by the left translation of $G/T$ by $G$.  
\end{thm}

Theorem \ref{thm:main1} is a generalization of the result 
\cite[Proposition 2.4]{N-S}, in which $T$ is  
assumed to be the maximal torus of $G$.    
We mention that the result due to Notbohm and Smith plays 
an important role in the proof of the uniqueness of fake Lie groups with
a maximal torus; see \cite[Section 1]{N-S1}.  
Theorem \ref{thm:main1} is deduced from Theorem \ref{thm:main2} below, which
gives a tractable criterion for the rational visibility. 

In order to describe Theorem \ref{thm:main2}, we fix notations. 
Let $G$ be a connected Lie group and $U$ a closed connected subgroup of
$G$. Let $B\iota : BU \to BG$ be the
 map induced by the inclusion $\iota : U \to G$. 
We assume that the rational cohomology of $BG$ is a polynomial algebra,
say $H^*(BG; \K)\cong \K[c_1, ..., c_k]$.  
In what follows, we write $H^*(X)$ for the cohomology of a space $X$
with coefficients in the rational field.

Consider the Lannes division functor 
$(H^*(BU) \! : \! H^*(G/U))$ in the category of differential graded
algebras (DGA's). 
Then the functor is regarded as a quotient of the free algebra  
$\wedge (H^*(BU)\otimes H_*(G/U))$, which in turn is isomorphic to 
$\wedge (QH^*(BU) \otimes H_*(G/U))$ as an algebra, where 
$QH^*(BU)$ denotes the vector space of indecomposable elements;  
see Section 2 for more details. 
Under the isomorphism, we can define an algebra map 
$u : (H^*(BU)  \! : \! H^*(G/U)) \to \K$ 
by 
$u(h\otimes b_*)= \langle j^*(h), b_* \rangle$, where $j : G/U \to BU$
is the fibre inclusion of the fibration 
$G/U \stackrel{j}{\to} BU \stackrel{B\iota}{\to} BG$. 
Moreover define $M_u$ to be the ideal of $(H^*(BU) \! : \! H^*(G/U))$ 
generated by the set 
$$
\{\eta | \deg \eta < 0 \} \cup \{\eta - u(\eta) | \deg \eta = 0\}. 
$$
Let $\pi : H^*(BU)\otimes H_*(G/U) \to (H^*(BU) \! : \! H^*(G/U))$
denote the composite of the inclusion 
$H^*(BU)\otimes H_*(G/U) \to \wedge (H^*(BU)\otimes H_*(G/U))$   
and the projection.

A recognition principle for rational visibility, Theorem \ref{thm:key}
below, enables one to deduce the following result.

\begin{thm}
\label{thm:main2}
With the above notation, assume that for 
$c_{i_1}, ..., c_{i_s} \in \{c_1, ..., c_k\}$, there are elements 
 $c_{j_1}, ..., c_{j_s} \in H^*(BG)$ 
and  $u_{1*}, ..., u_{s*} \in H^{\geq 1}(G/U)$ such that 
$$
\pi((B\iota)^*(c_{i_t})\otimes 1_*) \equiv 
\pi((B\iota)^*(c_{j_t})\otimes u_{t*})
$$ 
for $t = 1, ..., s$ modulo decomposable elements in 
$(H^*(BG) \! : \! H^*(G/U))/M_u$. Then there exists a map 
$\rho : \times_{j=1}^sS^{\deg c_{i_j}-1} \to G$ such that 
$\times_{j=1}^sS^{\deg c_{i_j}-1}$ 
is rationally visible in $\text{\em aut}_1(G/U)$ with respect to the
 map $(\lambda_{G,G/U})\circ \rho$. 
In particular, 
if $(B\iota)^*(c_{i_1}), ..., (B\iota)^*(c_{i_s})$ are decomposable
 elements, then 
$\pi((B\iota)^*(c_{i_t})\otimes 1_*)\equiv 0$  in 
$(H^*(BG)  \! : \! H^*(G/U))/M_u$ for $t = 1, ..., s$ 
and hence one obtains the same conclusion. 
\end{thm}

For a Lie group $G$ and a homogeneous space $M$ which admits a left
translation  by $G$, put 
$n(G):=\{ i \in {\mathbb N} \ | \ \pi_i(G)\otimes \K \neq 0\}$
and define the set $vd(G, M)$ of {\it visible degrees} by 
$$
vd(G, M) = 
\{ i \in n(G) \ | \ (\lambda_{G,M})_* : \pi_i(G)\otimes \K
\to  \pi_i(\aut{M})\otimes \K \ \text{is injective} \}.
$$

As for monoids of homeomorphisms and of diffeomorphisms, we benefit by
the study of rational visibility. In fact, we have an immediate but
very important corollary.  

\begin{cor}
\label{cor:Diff} 
If $l \in vd(G, M)$, then 
there exists an element with infinite order in  
$\pi_l(\text{\em Diff}_1(M))$  and 
$\pi_l(\text{\em Homeo}_1(M))$.   
\end{cor}

\begin{ex}
\label{ex:spheres}
Since $SO(d+1)/SO(d)$ is homeomorphic to the sphere $S^{d}$, 
we can define the maps 
$\lambda_{SO(d+1), S^{d}} : SO(d+1) \to \aut{S^{d}}$  
by left translations.  
The Haefliger, Brown and Szczarba model for the function space 
$\text{aut}_1(S^{d})$ allows us to deduce that 
$\aut{S^{2m+1}}\simeq_\K S^{2m+1}$ and 
$\aut{S^{2m}}\simeq_\K S^{4m-1}$; see Example \ref{ex:evaluation-map}
 below. 
Therefore $\lambda_{SO(d+1), S^{d}}$ is not 
 injective on the rational homotopy in general. However it follows that 
the induced maps 
$$(\lambda_{SO(2m+2), S^{2m+1}})_* : 
\pi_{2m+1}(SO(2m+2))\otimes \K 
\to \pi_{2m+1}(\text{aut}_1(S^{2m+1}))\otimes \K,  
$$
$$
(\lambda_{SO(2m+1), S^{2m}})_* : 
\pi_{4m-1}(SO(2m+1))\otimes \K \to \pi_{4m-1}(\text{aut}_1(S^{2m}))\otimes \K
$$  
are injective. In fact it is well known that, as algebras, 
$H^*(BSO(2m+1))\cong \K[p_1, ..., p_m]$ and 
$H^*(BSO(2m+2))\cong \K[p_1, ..., p_m, \chi]$, where 
$\deg p_j = 4j$ and $\deg \chi = 2m+2$. 
Moreover for inclusions $\iota_1 : SO(2m+1) \to SO(2m+2)$ and  
$\iota_2 : SO(2m) \to SO(2m+1)$, 
we see that $(B\iota_1)^*(\chi)=0$ and $(B\iota_2)^*(p_m)= \chi^2$; 
see \cite{M-T}. 
Thus Theorem  \ref{thm:main2} yields that  
$vd(SO(2m+2), S^{2m+1})=\{ 2m+1 \}$ and that    
$vd(SO(2m+1), S^{2m})=\{ 4m-1 \}$.

The result \cite[1.1.5 Lemma]{A-B-K} allows one to conclude that 
the map $SO(d+1) \to \text{Diff}_1(S^{d})$ induced by the left
 translations is injective on the homotopy group. 
This implies that the inclusion $\text{Diff}_1(S^{d}) \to \aut{S^d}$ is
 surjective on the rational homotopy group.    
\end{ex}

Theorem \ref{thm:key} yields another proof of 
results due to Kedra and McDuff \cite{K-M} and Sasao \cite{S}.

\begin{thm}\cite[Proposition 4.8]{K-M}\cite{S}
\label{thm:main3}
Let $M$ be the flag manifold of the form  
$U(m)\left/ U(m_1)\times \cdots \times U(m_l) \right.$. 
Then $SU(m)$ is rationally visible in $\text{\em aut}_1(M)$ with respect to 
the map $\lambda_{SU(m), M}$ given by the left translations; that is, 
$vd(SU(m), M)=n(SU(m))=\{3, 5, ..., 2m-1\}$. 
In particular,   
the localized map 
$$(\lambda_{SU(m), U(m)/ U(m-1)\times U(1)})_\K 
: SU(m)_\K \to {\text{\em aut}_1({\mathbb C}P^{m-1})}_\K
$$
is a homotopy equivalence.  
\end{thm}

Furthermore, the same argument as in the proof of 
Theorem \ref{thm:main3} allows one to establish the following result. 

\begin{thm}
\label{thm:main4}
Let $M$ be the flag manifold 
$Sp(m)\left/ Sp(m_1)\times \cdots \times Sp(m_l) \right.$. 
Then $vd(Sp(m), M)=\{7, 11, ..., 4m-1\}$. 
In particular, the $3$-connected cover $Sp(m)\langle 3 \rangle$ is 
rationally visible in 
$\text{\em aut}_1(M)$ with respect to $\lambda_{Sp(m),M}\circ \pi$, 
where $\pi : Sp(m)\langle 3 \rangle \to Sp(m)$ is the projection.  
\end{thm}

Let $G$ be a compact connected simple Lie group and $U$ a closed 
connected subgroup for which $G/U$ is a simply-connected homogeneous
space of rank one; that is, its rational cohomology is
 generated by a single element. In order to illustrate usefulness of
 Theorems \ref{thm:main2} and \ref{thm:key}, by applying the results, 
we determine visible degrees of $G$ in $\aut{G/U}$ for each couple $(G, U)$ 
classified by Oniscik in \cite[Theorems 2 and 4]{O}.

In the following table, we first list such homogeneous spaces 
of the form $G/U$ with a simple Lie group $G$ 
and its subgroup $U$, which is not diffeomorphic to spheres or
 projective spaces, together with the sets $vd(G, G/U)$.

\medskip
\hspace{-0.4cm}
{\small
\begin{tabular}{|l|c|l|l|}
\hline
\qquad \ \ \ $(G, U, \text{index})$ & $(G/U)_\K$ 
  & $vd(G, G/U)$ & $n(G)$ \\ \hline
(1) $(SO(2n+1), SO(2n-1)\!\times \!SO(2), 1)$ & ${\mathbb C}P^{2n-1}$ & 
 $\{3, ..., 4n-1\}$ & $\{3, ..., 4n-1\}$ \\ 
(2)  $(SO(2n+1), SO(2n-1), 1)$ 
 & $S^{4n-1}$ & $\{4n-1\}$ & $\{3, ..., 4n-1\}$ \\
(3)  $(SU(3), SO(3), 4)$ 
 & $S^{5}$ & $\{5\}$ & $\{3, 5\}$ \\
(4)  $(Sp(2), SU(2), 10)$ 
 & $S^{7}$ & $\{7\}$ & $\{3, 7\}$ \\
(5)  $(G_2, SO(4), (1,3))$ 
 & ${\mathbb H}P^{2}$ & $\{11\}$ & $\{3, 11\}$ \\
(6)  $(G_2, U(2), 3)$ 
 & ${\mathbb C}P^{5}$ & $\{3, 11\}$ & $\{3, 11\}$ \\
(7)  $(G_2, SU(2), 3)$ 
 & $S^{11}$ & $\{11\}$ & $\{3, 11\}$ \\
(6)'  $(G_2, U(2), 1)$ 
 & ${\mathbb C}P^{5}$ & $\{3, 11\}$ & $\{3, 11\}$ \\
(7)'  $(G_2, SU(2), 1)$ 
 & $S^{11}$ & $\{11\}$ & $\{3, 11\}$ \\
(8)  $(G_2, SO(3), 4)$ 
 & $S^{11}$ & $\{11\}$ & $\{3, 11\}$ \\
(9)  $(G_2, SO(3), 28)$ 
 & $S^{11}$ & $\{11\}$ & $\{3, 11\}$ \\
\hline
\end{tabular}
}
\begin{center}
Table 1
\end{center}
Here the value of the index of the inclusion $j : U \to G$ 
is regarded as the integer $i$ by which 
the induced map 
$j_* : H_3(U ; {\mathbb Z}) \to H_3(G ; {\mathbb Z})={\mathbb Z}$ 
is multiplication; see the proof of \cite[Lemma 4]{O}.   
The second column denotes the rational homotopy type of $G/U$
 corresponding a triple $(G, U, i)$. 
The homogeneous spaces $G/U$ for the cases (6)' and (7)' are
 diffeomorphic to those for the cases (1) and (2) with $n=3$,
 respectively.  Moreover, the homogeneous spaces are not diffeomorphic
 each other except for the cases (6)' and (7)'.

The following table describes visible degrees of a simple Lie group 
$G$ in $\aut{G/U}$ for which $G/U$ 
is of rank one and 
diffeomorphic to the sphere or the projective space, where the second
 column denotes the diffeomorphism type of the homogeneous space 
$G/U$ for the triple $(G, U, i)$.    

\medskip
\hspace{-0.4cm}
{\small
\begin{tabular}{|l|c|l|l|}
\hline
\qquad \ \ \ $(G, U, \text{index})$ & $G/U$ 
  & $vd(G, G/U)$ & $n(G)$ \\ \hline
(10) $(SU(n+1), SU(n), 1)$ & $S^{2n+1}$ & 
 $\{2n+1\}$ & $\{3, ..., 2n+1\}$ \\ 
(11)  $(SU(n+1), S(U(n)\!\times \!U(1)), 1)$ & ${\mathbb C}P^{n}$ & 
 $\{3, ..., 2n+1\}$ & $\{3, ..., 2n+1\}$ \\ 
(12)  $(SO(2n+1), SO(2n), 1)$ 
 & $S^{2n}$ & $\{4n-1\}$ & $\{3, ..., 4n-1\}$ \\
(13)  $(SO(9), SO(7), 1)$ 
 & $S^{15}$ & $\{15\}$ & $\{3, 7, 11, 15\}$ \\
(14)  $(Spin(7), G_2, 1)$ 
 & $S^7$ & $\{7\}$ & $\{3, 7 ,11\}$ \\
(15)  $(Sp(n), Sp(n-1), 1)$ 
 & $S^{4n-1}$ & $\{4n-1\}$ & $\{3, ..., 4n-1\}$ \\
(16)  $(Sp(n), Sp(n-1)\times S^1, 1)$ 
 & ${\mathbb C}P^{2n-1}$ & $\{3, ..., 4n-1\}$ & $\{3, ..., 4n-1\}$ \\
(17)  $(Sp(n), Sp(n-1)\!\times \! Sp(1), 1)$ 
 & ${\mathbb H}P^{n-1}$ & $\{7, ..., 4n-1\}$ & $\{3, ..., 4n-1\}$ \\
(18)  $(SO(2n), SO(2n-1), 1)$ 
 & $S^{2n-1}$ & $\{2n-1\}$ & $\{3, ..., 4n-5, 2n-1\}$ \\
(19)  $(F_4, Spin(9), 1)$ 
 & ${\mathcal L}P^2$ & $\{23\}$ & $\{3, 11, 15, 23\}$ \\
(20)  $(G_2, SU(3), 1)$ 
 & $S^{6}$ & $\{11\}$ & $\{3, 11\}$ \\
\hline
\end{tabular}
}
\begin{center}
Table 2
\end{center}

\noindent
Here ${\mathcal L}P^2$ stands for the Cayley plane. 

The former half of Theorem \ref{thm:main2}, namely the
Lannes functor argument, does work well enough when determining the set 
$vd(G_2, G_2/U(2))$ of visible degrees 
in case (6) in Table 1; see Section 8. 
Observe that for the cases (12) and (18) 
the results follow from those in Example \ref{ex:spheres}. We are aware 
 that in the above tables 
$G$ is rationally visible in $\aut{G/U}$ if and only if $G/U$ has the
rational homotopy type of 
 the complex projective space. It should be mentioned that 
for the map 
$\lambda_* : \pi_*(F_4)\otimes \K \to 
\pi_*(\aut{{\mathcal L}P^2})\otimes \K$, 
the restriction $(\lambda_*)_{15}$ is not injective though the vector space 
$\pi_{15}(\text{aut}_1({\mathcal L}P^2))\otimes \K$ and
 $\pi_{15}(F_4)\otimes \K$ are non-trivial; see Section 8. 
Moreover, Corollary \ref{cor:Diff} enables us to obtain 
non-trivial elements with infinite order in 
$\pi_l(\text{Diff}_1(M))$ and 
$\pi_l(\text{Homeo}_1(M))$ for each homogeneous space $M$ described in
Tables 1 and 2.

Let $X$ be a space and ${\mathcal H}_{H, X}$ the monoid of all homotopy
equivalences that act trivially on the rational homology of $X$. 
The result \cite[Proposition 4.8]{K-M}  
asserts that if $X$ is generalized flag manifold 
$U(m)\left/ U(m_1)\times \cdots \times U(m_l) \right.$, 
then the map  $B\psi_{SU(m)} : BSU(m) \to B{\mathcal H}_{H, X}$ 
arising from the left translations is injective on the rational homotopy.    
Let $\iota : \text{aut}_1(X) \to {\mathcal H}_{H, X}$ be the inclusion.   
Since  $B\psi_{SU(m)} = B\iota\circ B\lambda_{SU(m),X}$,   
the result \cite[Proposition 4.8]{K-M} yields Theorem \ref{thm:main3}.  
Theorem \ref{thm:main5} below guarantees that the converse also
holds; that is, the result due to Kedra and McDuff   
is deduced from Theorem \ref{thm:main3}; see Section 7. 

Before describing Theorem \ref{thm:main5}, 
we recall an $F_0$-space, which is a
simply-connected finite complex with finite-dimensional rational
homotopy and trivial rational cohomology in odd degree. For example, 
a homogeneous space $G/T$ for which $G$ is a connected Lie group and 
$T$ is a maximal torus of $G$ is an $F_0$-space.  

\begin{thm}
\label{thm:main5}
Let $X$ be an $F_0$-space or a space having the rational homotopy type of the
 product of odd dimensional 
 spheres and $G$ a connected topological group which acts on 
 $X$. 
Then $(B\lambda_{G, X})_* :  H_*(BG) 
\to H_*(B\text{\em aut}_1(X))$ is
 injective if and only if so is 
$(B\psi)_* : H_*(BG) \to H_*(B{\mathcal H}_{H, X})$. Here 
$\psi : G \to {\mathcal H}_{H, X}$ denotes the morphism of monoids induced
 by the action of $G$ on $X$.   
\end{thm}

We now provide an overview of the rest of the paper. 
In Section 2, 
we recall a model for the evaluation map of a function space 
from \cite{Ku1}, \cite{B-M} and \cite{H-K-O}. 
In Section 3,
a rational model for the map $\lambda_{G,M}$ mentioned above is
constructed. Section 4 is devoted
to the study of a model for the left translation of a Lie group on a
homogeneous space. 
In Section 5, we prove 
Theorem \ref{thm:main2}.  
Theorem \ref{thm:main3} is proved in Section 6. 
In Section 7, we prove Theorem \ref{thm:main5}. 
 The results on visible degrees in Tables 1 and 2 are verified
in Section 8. 
In Appendix, Section 9, the group cohomology of 
$\text{Diff}_1(M)$ for an appropriate homogeneous space $M$ is
discussed. 
By using Theorem \ref{thm:main2}, we find a non-trivial class in the
group cohomology.

\section{Preliminaries}

The tool for the study of the rational visibility problem 
is a rational model for the evaluation map $ev : \aut{M}\times M \to M$, 
which is described in terms of the rational model due to Brown and
Szczarba \cite{B-S} and Haefliger \cite{H}. 
For the convenience of the reader and to make notation more precise, 
we recall from \cite{B-M} and \cite{Ku1} the model for the evaluation map.    
We shall use the same terminology as in \cite{B-G} and \cite{F-H-T}. 

Throughout the paper, for an augmented algebra $A$, we write $QA$ for
the space $\overline{A}/\overline{A}\cdot \overline{A}$ 
of indecomposable elements, where $\overline{A}$ denotes the
augmentation ideal. For a DGA $(A, d)$, let $d_0$ denote the linear part
of the differential.

In what follows, we assume that a space is nilpotent and 
has the homotopy type of a connected 
CW complex with rational homology of finite type 
unless otherwise explicitly stated. 
We denote by $X_\K$ the localization of a nilpotent space $X$. 

Let $A_{PL}$ be the simplicial commutative cochain algebra of polynomial
differential forms with coefficients in ${\mathbb Q}$; see
\cite{B-G} and  \cite[Section 10]{F-H-T}.  
Let $\mathcal{A}$ and  $\Delta {\mathcal S}$ be the category of DGA's
and that of simplicial sets, respectively.  
Let $\mbox{DGA}(A, B)$ and $\mbox{Simpl}(K, L)$ denote the hom-sets of
the categories $\mathcal{A}$ and $\Delta {\mathcal S}$, respectively.    
Following Bousfield and Gugenheim \cite{B-G}, we define functors 
$\Delta : \mathcal{A} \to  \Delta {\mathcal S}$ 
and $\Omega : \Delta {\mathcal S} \to  \mathcal{A}$
by $ \Delta (A)= \mbox{DGA}(A, A_{PL})$ and by 
$\Omega (K) = \mbox{Simpl}(K, A_{PL})$.   

Let $(B, d_B)$ be a connected, locally finite DGA and $B_*$ denote
the differential graded coalgebra defined by $B_q=\mbox{Hom}(B^{-q}, \K)$
for $q\leq 0$ together with the coproduct $D$ and the differential $d_{B
*}$ which are dual to the multiplication of $B$ and to the differential
$d_B$, respectively. 
We denote by $I$ the ideal of the free algebra 
$\wedge(\wedge V \otimes B_*)$ 
generated by $1\otimes 1_* -1$ and all elements of the form 
$$
a_1a_2\otimes \beta - 
\sum_i(-1)^{|a_2||\beta_i'|}(a_1\otimes \beta_i')(a_2\otimes \beta_i''), 
$$
where $a_1, a_2 \in \wedge V$, $\beta \in B_*$ and $D(\beta) 
= \sum_i\beta'_i\otimes  \beta''_i$. 
Observe that $\wedge(\wedge V \otimes B_*)$ is a DGA with the differential 
$d := d_A\otimes 1\pm 1\otimes d_{B *}$. 
The result \cite[Theorem 3.5]{B-S} implies that the composite 
$
\rho : \wedge(V \otimes B_*)  \hookrightarrow 
\wedge(\wedge V \otimes B_*) \to 
\wedge(\wedge V \otimes B_*)/I
$ 
is an isomorphism of graded algebras. Moreover, it follows that
\cite[Theorem 3.3]{B-S} that $dI \subset I$.   
Thus $(\wedge(V \otimes B_*), \delta = \rho^{-1}d\rho)$ is a DGA. 
Observe that, for an element $v \in V$ and a cycle $e \in B_*$,  
if  $d(v) = v_1\cdots v_m$ with $v_i \in V$ and 
$D^{(m-1)}(e_j) = \sum_je_{j_1}\otimes \cdots \otimes e_{j_m}$,
then
$$
\begin{array}{lcl}
\delta(v\otimes e) 
&=&\sum_j\pm (v_1\otimes e_{j_1}) \cdots (v_m\otimes e_{j_m}).  
\end{array}
\eqnlabel{add-1}
$$
Here the sign is determined by the Koszul rule; that is, 
$ab = (-1)^{\deg a \deg b} ba$  in a graded algebra.  
Let $F$ be the ideal of $E:=\wedge(V \otimes B_*)$ generated by 
$\oplus_{i < 0}E^i$ and $\delta (E^{-1})$. Then $E/F$ 
is a free algebra and $(E/F, \delta)$ is a Sullivan algebra 
(not necessarily connected), see the proofs of \cite[Theorem 6.1]{B-S}
and of \cite[Proposition 19]{B-M}.  

\begin{rem}
\label{rem:Lannes}
The result \cite[Corollary 3.4]{B-S} implies that there exists a natural
 isomorphism 
$
\text{DGA}(\wedge (\wedge V\otimes B_*)/I, C) \cong 
\text{DGA}(\wedge V, B\otimes C)
$
for any DGA $C$. Then $\wedge (\wedge V\otimes B_*)/I$ is regarded as
 the Lannes division functor $(\wedge V  \! : \! B)$ by definition. 
\end{rem}

The singular simplicial set of a topological space $U$ 
is denoted by $\Delta U$ 
and let $|X|$ be the geometrical realization of a simplicial
set $X$. By definition, $A_{PL}(U)$ the DGA of polynomial differential
forms on $U$ is given by $A_{PL}(U)= \Omega \Delta U$. 
Given spaces $X$ and $Y$, we denote by $\F(X, Y)$ the space of continuous
maps from $X$ to $Y$. The connected component of $\F(X, Y)$ containing a
map $f : X \to Y$ is denoted by $\F(X, Y; f)$.

Let 
$\alpha : A=(\wedge V, d)\stackrel{\simeq}{\to} A_{PL}(Y)=\Omega
\Delta Y$  be a Sullivan model (not necessarily minimal) 
for $Y$ and $\beta : (B, d) \stackrel{\simeq}{\to} A_{PL}(X)$ 
a Sullivan model for $X$ for which $B$ is connected and locally finite.  
For the function space $\F(X, Y)$ which is considered below,   
we assume that
%
$$
\dim \oplus_{q\geq 0}H^q(X; \K)< \infty \quad \text{or} \quad  
\dim \oplus_{i\geq 2}\pi_i(Y)\otimes \K < \infty.
\eqnlabel{add-2}
$$
Then the proof of \cite[Proposition 4.3]{Ku1} enables us to deduce the
following lemma; see also \cite{B-M}.  

\begin{lem}
\label{lem:key-1} 
{\em (i)} Let $\{b_j\}$ and $\{b_{j*}\}$ be 
a basis of $B$ and its dual basis of  $B_*$, respectively and 
$\widetilde{\pi} : \wedge(A\otimes B_*) \to 
(\wedge(A\otimes B_*)/I) \bigl/F \bigr.$ denote the projection. 
Define a map  
$m(ev) : A \to (\wedge(A\otimes B_*)/I)\bigl/F \bigr. \otimes B$ 
by  
$$m(ev)(x)= \sum_{j}(-1)^{\tau(|b_j|)}
\widetilde{\pi}(x\otimes b_{j *})\otimes b_j,    
$$
for $x \in A$, where 
$\tau(n)=[(n+1)/2]$, the greatest integer in  $(n+1)/2$.  
Then $m(ev)$ is a well-defined DGA map. \\
{\em (ii)} 
There exists a commutative diagram 
$$
\xymatrix@C25pt@R15pt{
\F(X_\K, Y_\K) \times X_\K  \ar[r]^(0.7){ev} & Y_\K \\
|\Delta(E/F)| \times |\Delta(B)| 
\ar[u]^{\Theta  \times 1} \ar[r]_(0.66){|\Delta m(ev)|} & 
 |\Delta(A)| \ar@{=}[u]
}
$$
in which $\Theta$ is the homotopy equivalence described in 
\cite[Sections 2 and 3]{B-S}; see also \cite[(3.1)]{Ku1}.  
\end{lem}



We next recall a Sullivan model for a connected component 
of a function space. 
Choose a basis $\{a_k', b_k',  c_j' \}_{k,j}$ 
for $B_*$ so that $d_{B_*}(a_k')=b_k'$, 
$d_{B_*}(c_j')=0$ and $c_0' = 1$.
Moreover we take a basis $\{v_i\}_{i\geq 1}$ for $V$ such that
$\deg v_i \leq \deg v_{i+1}$ and $d(v_{i+1})\in \wedge V_i$, where
$V_i$ is the subvector space spanned by the elements $v_1 ,..., v_i$.
The result \cite[Lemma 5.1]{B-S} ensures that there exist free algebra
generators $w_{ij}$, $u_{ik}$ and $v_{ik}$ such that  

(2.3) $w_{i0}= v_i\otimes 1$  and
$w_{ij}= v_i \otimes c_j' + x_{ij}$, where $x_{ij}\in \wedge(V_{i-1} \otimes
B_*)$,

(2.4) $\delta w_{ij}$ is 
in $\wedge(\{w_{sl} ; s < i\})$,

(2.5) $u_{ik}= v_i\otimes a_k'$ and ${\delta}u_{ik}=v_{ik}$.

\noindent
We then have a inclusion 
$$
\gamma : E:=(\wedge(w_{ij}), \delta) \hookrightarrow  
(\wedge(V\otimes B_*), \delta), 
\leqno{(2.6)}
$$ 
which is a homotopy equivalence with a retract 
$$
r :  (\wedge(V\otimes B_*), \delta) \to E ;
\leqno{(2.7)} 
$$
see \cite[Lemma 5.2]{B-S} for more details. 
Let $q$ be a Sullivan representative for a map $f : X \to Y$; that is,
$q$ fits into the homotopy commutative diagram 
$$
\xymatrix@C30pt@R15pt{
\wedge W \ar[r]^(0.45){\simeq} & A_{PL}(X) \\
 \wedge V \ar@{>}[u]^q \ar[r]_(0.45){\simeq} & A_{PL}(Y) . 
  \ar[u]_{A_{PL}(f)}
}
$$
Moreover we define a $0$-simplex  
$\widetilde{u} \in \Delta(\wedge(\wedge V\otimes B_*)/I)_0$ 
by
$$
\widetilde{u}(a\otimes b)=(-1)^{\tau(|a|)}b(q(a)), 
\leqno{(2.8)} 
$$
where $a \in \wedge V$ and $b \in B_*$. Put $u= \Delta(\gamma)\widetilde{u}$.
Let $M_u$ be the ideal of $E$ generated by the set
$
\{\eta \mid \deg \eta < 0\} \cup \{\delta\eta \mid \deg \eta =
0\}\cup
\{\eta - u(\eta )\mid \deg \eta  = 0\}.
$
Then we see that  $(E/M_u, \delta)$ is an explicit model for the
connected component $F(X, Y ; f)$; see \cite[Theorem 6.1]{B-S} and 
\cite[Section 3]{H-K-O}. 
The proof of \cite[Proposition 4.3]{Ku1} and \cite[Remark 3.4]{H-K-O}
allow us to deduce the following proposition; see also \cite{B-M}.  

\begin{prop}
\label{prop:ev-model} With the same notation as in Lemma
 \ref{lem:key-1}, we define a map  
$m(ev) : A=(\wedge V, d) \to  (E/M_u, \delta)\otimes B$ 
by  
$$m(ev)(x)= \sum_{j}(-1)^{\tau(|b_j|)}
\pi\circ r(x\otimes b_{j *})\otimes b_j,    
$$
for $x \in A$, where $\pi : E \to E/M_u$ denotes the natural projection.  
Then $m(ev)$ is a model for the evaluation map 
$ev : \F(X, Y; f) \times X \to Y$; that is, there exists a homotopy
 commutative diagram 
$$
\xymatrix@C30pt@R15pt{
A_{PL}(Y) \ar[r]^(0.4){A_{PL}(ev)}  &  A_{PL}(\F(X, Y; f) \times X)  \\
  & A_{PL}(\F(X, Y; f))\otimes A_{PL}(X) \ar[u]_{\simeq}\\
 A \ar@{>}[uu]_{\simeq}^{\alpha} \ar[r]_(0.4){m(ev)} & 
 (E/M_u, \delta)\otimes B,  \ar[u]^{\simeq}_{\xi \otimes \beta}
}
$$
in which 
$\xi : (E/M_u, \delta) \stackrel{\simeq}{\to} A_{PL}(\F(X,Y; f))$ 
is the Sullivan model for $\F(X,Y; f)$ due to Brown and Szczarba
 \cite{B-S}.   
\end{prop}

We call the DGA $(E/M_u, \delta)$ the Haefliger-Brown-Szczarba model 
(HBS-model for short) 
for the function space  $\F(X,Y; f)$. 

\begin{ex}
\label{ex:evaluation-map} 
Let $M$ be a space whose rational cohomology is isomorphic to the
 truncated algebra $\K[x]/(x^m)$, where $\deg x = l$.   
Recall the model  $(E/M_u, \delta)$ 
for $\text{aut}_1(M)$ mentioned in \cite[Example 3.6]{H-K-O}. 
Since the minimal model for $M$ has the form 
$(\wedge (x, y), d)$ with $dy=x^m$, it follows that 
$$
E/M_u 
= \wedge (x\otimes 1_*, y\otimes(x^s)_* ; 0\leq s\leq m-1)
$$ 
with $\delta(x\otimes 1_*)= 0$ and  
$\delta(y\otimes(x^s)_*)=(-1)^{s}
\left(
\begin{array}{c}
m \\
s
\end{array}
\right)
(x\otimes 1_*)^{m-s}$, where 
$\deg x\otimes 1_* =l$ and $\deg (y\otimes (x^s)_*) =  lm-ls-1$.   
Then the rational model $m(ev)$ for the evaluation map 
$ev : \aut{M} \times M 
\to M$ 
is given by 
$m(ev)(x)=(x\otimes 1_*)\otimes 1 + 1\otimes x$ and 
$$
m(ev)(y)= \sum_{s=0}^{m-1}(-1)^s(y\otimes (x^s)_*)\otimes x^s + 
1\otimes y. 
$$
\end{ex}

\begin{rem}
\label{rem:variants} 
We here describe variants of the HBS-model for a function space. \\
(i) Let $\wedge \widetilde{V} \stackrel{\simeq}{\to} A_{PL}(Y)$ be a
Sullivan model (not necessarily minimal) and 
$B \stackrel{\simeq}{\to} A_{PL}(X)$ a Sullivan model of finite
type. 
We recall the homotopy equivalence 
$\gamma : E \to \widetilde{E}=\wedge(\wedge V\otimes B_*)/I$ mentioned
in (2.6).  
Let $\widetilde{u} \in \Delta(\widetilde{E})_0$ be a $0$-simplex and 
$u$ a $0$-simplexes of $E$ 
defined by composing $\widetilde{u}$ with the quasi-isomorphism $\gamma$.  
Then the induced map 
$\overline{\gamma} : E/M_u \to \widetilde{E}/M_{\widetilde{u}}$ 
is a quasi-isomorphism.  
In fact the results \cite[Theorem 6.1]{B-S} and  \cite[Proposition 19]{B-M}
imply that the projections onto the quotient DGA's $E/M_u$ and 
$\widetilde{E}/M_{\widetilde{u}}$ induce homotopy equivalences  
$\Delta(p) : \Delta(E/M_u) \to \Delta(E)_{u}$ and  
$\Delta(\widetilde{p}) : \Delta(\widetilde{E}/M_{\widetilde{u}}) \to 
\Delta(\widetilde{E})_{\widetilde{u}}$, respectively. 
Here $K_v$ denotes the connected component containing the vertex $v$ for
 a simplicial set $K$, namely, the set of simplices all of whose faces
 are at $v$. 

Then we have a commutative diagram  
$$
 \xymatrix@C30pt@R20pt{
\pi_*(|\Delta(E/M_u)|) \ar[r]^{|\Delta (p)|}_{\cong} & \pi_*(|\Delta(E)|, |u|)
\\
\pi_*(|\Delta(\widetilde{E}/M_{\widetilde{u}})|) 
\ar[r]_{|\Delta (\widetilde{p})|}^{\cong} 
\ar[u]^{|\Delta(\overline{\gamma})|_*}
 & \pi_*(|\Delta(\widetilde{E})|, |\widetilde{u}|).  
\ar[u]_{|\Delta(\gamma)|_*}
}
$$
Since $\gamma$ is a homotopy equivalence, 
it follows that
$|\Delta(\gamma)|_*$ is an isomorphism and hence so is 
$|\Delta(\overline{\gamma})|_*$. 
This yields that $|\Delta(\overline{\gamma})|$ is 
a homotopy equivalence. By virtue of the
Sullivan-de Rham equivalence Theorem \cite[9.4]{B-G}, 
we see that $\overline{\gamma}$ is a quasi-isomorphism. 

As in Lemma \ref{lem:key-1}, we define the DGA map 
$\widetilde{m(ev)} : (\wedge V, d) \to
 \widetilde{E}/\widetilde{F}\otimes B$ and let 
$m(ev) : (\wedge V, d) \to \widetilde{E}/M_{\widetilde{u}}\otimes B$ be
 the DGA defined by $m(ev) = \pi\otimes 1 \circ \widetilde{m(ev)}$.   
 We then have a homotopy commutative diagram
$$
\xymatrix@C35pt@R5pt{
 & E/M_u \otimes B \ar[dd]^{\overline{\gamma}\otimes 1}_{\simeq}\\
\wedge V \ar[ru]^{m(ev)} \ar[rd]_{m(ev)} & \\
& \widetilde{E}/M_{\widetilde{u}}\otimes B. 
}
$$
In fact the homotopy between $id_{\widetilde{E}}$ and $\gamma\circ r$
 defined in \cite[Lemma 5.2]{B-S} induces a homotopy between 
$id_{\widetilde{E}/\widetilde{F}}$ and 
$\gamma \circ r : \widetilde{E}/\widetilde{F} \to E/F \to 
\widetilde{E}/\widetilde{F}$. 
It is immediate that $r\circ \gamma = id_{E/F}$.  
Let $m(ev)' : \wedge V \to E/F\otimes B$ be the DGA defined as in 
Proposition \ref{prop:ev-model}. Then it follows that 
\begin{eqnarray*}
\overline{\gamma}\otimes 1\circ m(ev) &=& 
\overline{\gamma}\otimes 1\circ \pi \otimes 1 \circ m(ev)' \\
&=& \pi \otimes 1\circ \gamma \otimes 1\circ r\otimes 1 \circ
 \widetilde{m(ev)} \\
&\simeq&  \pi \otimes 1\circ  \widetilde{m(ev)} =  m(ev). 
\end{eqnarray*}

\noindent
(ii) In the case where $X$ is formal, we have a more tractable model for
$\F(X, Y; f)$. Suppose that $X$ is a formal space with 
a minimal model $(B, d_B)=(\wedge W', d)$. 
Then there exists a quasi-isomorphism 
$k : (\wedge W', d) \to H^*(B)$ which is surjective; 
see \cite[Theorem 4.1]{DGMS}.  With the notation mentioned above, 
let $\{e_j\}_j$ be a basis for  
the homology $H(B_*)$ of the differential graded coalgebra $B_*=(\wedge W')_*$ 
and $\{v_i\}_i$ a basis for $V$.  
Then it follows from the proof of \cite[Theorem 1.9]{B-S} 
that the subalgebra $\K\{v_i\otimes e_j\}$ is closed for 
the differential $\delta$ and that 
the inclusion $\K\{v_i\otimes e_j\} \to \wedge (W\otimes B_*)=\widetilde{E}$ 
gives rise to a homotopy equivalence 
$$\gamma : E'
:=(\wedge (v_i\otimes e_j), \delta) \to 
(\wedge (W\otimes B_*), \delta)=\widetilde{E}.     
$$
In fact, the elements $w_{ij}$ in (2.3) can be chosen so that
 $w_{i0}=v_i\otimes 1_*$ and $w_{ij}= v_i\otimes e_j$ for $j \geq 1$.   
Moreover we see that there exists a retraction 
$r : \wedge (W\otimes B_*) \to E'$ which is the homotopy 
inverse of $\gamma$. 
Thus Proposition \ref{prop:ev-model} remains true after replacing $E$
by $E'$. 
Here the $0$-simplex 
$\widetilde{u} \in \Delta(\wedge(W \otimes B_*))_0$ 
needed in the construction of the model for $\F(X, Y; f)$ 
has the same form as in (2.8).  
\end{rem}

We conclude this section with some comments on models for a connected
component of a function space and related maps.  

In the original construction in \cite{H} and \cite{B-M} of a model for a
function space $\F(X, Y)$, 
it is assumed that the source space $X$ admits a finite dimensional
model. Indeed the construction of a model for the evaluation map in  
\cite[Theorem 1]{B-M} requires existence of such 
a model for the space $X$. 
As described in Lemma \ref{lem:key-1} and Proposition
\ref{prop:ev-model}, 
our construction only needs the assumption (2.2). 
Thus our model for a function space endowed with a model for evaluation
map is viewed as a generalization of that in \cite{B-M}.     

The arguments in \cite[Section 7]{B-S} and \cite{B-M} on a model for a
connected component of $\F(X, Y)$ begin with a $0$-simplex; that is, the
considered component is that containing a map $f$ which 
corresponds to the given $0$-simplex via a sequence of weak
equivalences between the singular simplicial set of $\F(X_\K, Y_\K)$ and the
simplicial set $\Delta(E/F)$; see \cite[(2.3)]{H-K-O}. 
On the other hand, for any given map $f : X \to Y$, an explicit form of 
a $0$-simplex corresponding to $f$ is clarified in 
\cite[Remark 3.4]{H-K-O} with (2.8). Thus our constructions in this
section complement the basic constructions in rational homotopy theory
of function spaces due to Buijs and Murillo \cite{B-M}.

Observe that $\aut{X}$ is 
nothing but the function space $\F(X, X; id_{M})$. 
Moreover, for a manifold $M$, the function space $\aut{M}$ satisfies the
assumption (2.2). Thus we have explicit models for $\aut{X}$ and for the
evaluation map according to the procedure in this section.  
Adding such the models, we moreover provide 
an elaborate model for the map $\lambda_{G, M}$ mentioned 
in Introduction in the next section; see Theorem \ref{thm:key} below.

\section{A rational model for the map $\lambda $ induced by 
left translation}

Let $M$ be a space admitting an action of Lie group $G$ on the left. 
We define the map $\lambda : G \to \text{aut}_1(M)$ by $\lambda(g)(x)=gx$. 
The subjective in this section is to construct an algebraic model for
the map 
$$
in \circ \lambda : G \to \text{aut}_1(M) \to \F(M, M),  
$$
where $in : \aut{M} \to  \F(M, M)$ denotes the inclusion. 
To this end we use a model for the evaluation map 
$$
ev : \F(X, Y)\times X \to Y
$$
defined by $ev(f)(x)= f(x)$ for $f \in \F(X, Y)$ and $x \in X$, which
is considered in \cite{Ku1} and \cite{B-M}.

Let $G$ be a connected Lie group, $U$ a closed subgroup of $G$ and $K$ a
closed subgroup which contains $U$.   
Let $(\wedge V_G, d)$ and $(\wedge W, d)$ denote a minimal model for 
$G$ and a Sullivan model for the homogeneous space $G/U$, respectively. 
Let $\lambda : G \to \F(G/U, G/K)$ be the adjoint of the composite 
of the left translation $G\times G/U \to G/U$ and
projection $p : G/U \to G/K$. Observe that the map 
$\lambda$ coincides with the composite 
$$
p_*\circ in \circ \lambda_{G, G/U} : G \to \text{aut}_1(G/U) 
\to \F(G/U, G/U) \to \F(G/U, G/K).  
$$
We shall construct a model for $\lambda$ 
by using the HBS-model for $\F(G/U, G/K; p)$ mentioned 
in Remark \ref{rem:variants} (i), a Sullivan representative  
$$
\zeta' : \wedge W \to \wedge V_G\otimes \wedge W'
$$
for the composite $G\times G/U \to G/K$ 
of the left translation  $G\times G/U \to G/U$ and the projection 
$p : G/U \to G/K$. 

Let $A$, $B$ and $C$ be connected DGA's.   
Recall from \cite[Section 3]{B-S} the bijection 
$\Psi : (A\otimes B_*, C)_{DG} \stackrel{\cong}{\to} 
(A, C\otimes B)_{DG}$
defined by 
$$
\Psi(w)(a) =  \sum_{j}(-1)^{\tau(|b_j|)}w(a\otimes b_{j *})\otimes b_j. 
$$ 
Consider the  case where $A=(\wedge W, d)$, $B=(\wedge W', d)$ 
and $C=(\wedge V_G, d)$. 
Moreover define a map 
$
\widetilde{\mu} : \wedge (A\otimes B_*) \to \wedge V_G
$
by 
$$
\widetilde{\mu}(y \otimes b_{j *}) 
= (-1)^{\tau(|b_j|)} \langle \zeta' (y), b_{j *} \rangle,  
\eqnlabel{add-0}
$$
where $\langle \ \ , b_{j *} \rangle : \wedge V_G \otimes \wedge W' \to 
\wedge V_G$ is a map defined by $\langle x \otimes a , b_{j *} \rangle =
x\cdot \langle a,b_{j *} \rangle$.  
Then we see that $\Psi(\widetilde{\mu})= \zeta'$. Hence it follows from
\cite[Theorem 3.3]{B-S} that 
$$
\widetilde{\mu} : \widetilde{E}:= \wedge (A\otimes B_*)/I \to \wedge V_G 
$$
is a well-defined DGA map. 
We define an augmentation $\widetilde{u} : \widetilde{E} \to \K$ by 
$\widetilde{u} = \varepsilon \circ \widetilde{\mu}$, where 
$\varepsilon : \wedge V_G \to \K$ is the augmentation. It is readily
seen that that $\widetilde{\mu}(M_{\widetilde{u}})=0$. 
Thus we see that 
$\widetilde{\mu}$ induces a DGA map  
$\widetilde{\widetilde{\mu}} : \widetilde{E}/M_{\widetilde{u}} 
\to \wedge V_G$.

The result \cite[Theorem 3.11]{H-M-R} asserts that the map  
$$
e_\sharp : \F(G/U, (G/K); p) \to\F(G/U, (G/K)_\K; e\circ p)
$$ is a localization. Thus we have the localization     
$\lambda_\K : G_\K \to \F( G/U, (G/K)_\K; e\circ p)$. Observe that 
$\lambda_\K$ fits into the homotopy commutative diagram 
$$
\xymatrix@=18pt{
G_\K \ar[r]^(0.25){\lambda_\K} & \F( G/U, (G/K)_\K; e\circ p) \\
G \ar[u]^{e} \ar[r]_(0.25){\lambda} & 
\F( G/U, (G/K); p), \ar[u]_{e_{\sharp}}
}
$$
where $e$ denotes the localization map. 

We then have a recognition principle for rational visibility.

\begin{thm}
\label{thm:key}
Let $\{x_i\}_i$ be a basis for the image of the induced map 
$$
H^*(Q(\widetilde{\widetilde{\mu}})) : 
H^* (Q(\widetilde{E}/M_{\widetilde{u}}), \delta_0) 
\to H^*(Q(\wedge V_G), d_0)=V_G. 
$$ 
Then there exists a map $\rho : \times_{j=1}^sS^{\deg x_i} \to G$ such 
that the map 
$$
(\lambda_\K \circ \rho_\K)_* : 
\pi_*((\times_{j=1}^sS^{\deg x_i})_\K) \to \pi_*(\F(G/U, (G/K)_\K),
 e\circ p)
$$
is injective.  
Moreover $(\lambda_\K)_* : \pi_i(G_\K) \to \pi_i(\F(G/U, (G/K)_\K), e\circ p)$
is injective if and only if $H^i(Q(\widetilde{\widetilde{\mu}}))$ is
 surjective. 
\end{thm}

In order to prove Theorem \ref{thm:key}, 
it suffices to show that $\widetilde{\widetilde{\mu}}: 
\widetilde{E}/M_{\widetilde{u}} \to \wedge V_G$ is a Sullivan model for
the map $G_\K \to \F(G/U, (G/K)_\K; e\circ p)$. 

We first observe that the diagram 
$$
\xymatrix@C15pt@R15pt{
\wedge V_G \otimes \wedge W'  
     & & (\wedge (A\otimes B_*)/I)/F \otimes \wedge W'
 \ar[ll]_(0.6){\widetilde{\mu}\otimes 1} \\
& \wedge W \ar[ul]^{\zeta'} \ar[ur]_{m(ev)} &
}
\eqnlabel{add-2}
$$ 
is commutative.  Thus Lemma \ref{lem:key-1} enables us to obtain a 
commutative diagram 
$$
\xymatrix@C12pt@R15pt{
|\Delta\wedge V_G| \times  |\Delta\wedge W'| 
\ar[dr]_{|\Delta\zeta'|=\text{action}_\K} 
     \ar[rr]^{(\Theta \circ |\Delta\widetilde{\mu}|) \times 1} & & 
 \F((G/U)_\K, (G/K)_\K )\times (G/U)_\K 
 \ar[dl]^{ev} \\ 
& |\Delta\wedge W| = (G/K)_\K . 
}
\eqnlabel{add-2}
$$
%
%
%
Observe that the assumption (2.2) is satisfied in the case where we
here consider.  

Since the restriction $|\Delta\zeta'|_{|_{*\times |\Delta\wedge W|}}$ 
is homotopic to $p_\K$, it
follows from the commutativity of the diagram (3.3) that 
$p_\K \simeq \Theta \circ |\Delta\widetilde{\mu}|(*)$. This implies that 
$\Theta \circ |\Delta\widetilde{\mu}|$ maps $G_\K$ into the function space 
$\F((G/U)_\K, (G/K)_\K ; p_\K)$.

\begin{lem}
\label{lem:homotopy} Let $\lambda_\K : G_\K \to \F( G/U, (G/K)_\K;
 e\circ p)$ be
 the localized map of $\lambda$ mentioned above and 
$e^\sharp : \F((G/U)_\K, (G/K)_\K ; p_\K) \to 
 \F((G/U), (G/K)_\K ; e\circ p)$ the map induced by the localization 
$e : (G/U) \to (G/U)_\K$. Then 
$$  
e^\sharp \circ \Theta \circ |\Delta\widetilde{\mu}|\simeq \lambda_\K : 
G_\K \to \F((G/U), (G/K)_\K ; e\circ p). 
$$
\end{lem}

\begin{proof}
Consider the commutative diagram 
$$
\xymatrix@=18pt{
[G\times G/U, G/K] \ar[d]_{e_*} 
 \ar[r]^{\theta}_{\approx} & [G, \F(G/U, G/K)] \ar[d]^{(e_\sharp)_*} \\
  [G\times G/U, (G/K)_\K]  \ar[r]^{\theta}_{\approx} & 
               [G, \F(G/U, (G/K)_\K)]  \\
[G_\K \times (G/U)_\K, (G/K)_\K] \ar[u]^{(e\times e)^*}_{\approx}  
 \ar[dr]^{\theta}_{\approx} 
            & [G_\K, \F(G/U, (G/K)_\K] \ar[u]_{e^*}^{} \\
   & [G_\K, \F((G/U)_\K, (G/K)_\K)] \ar[u]_{(e^\sharp)_*}^{\approx}
}
\eqnlabel{add-3}
$$ 
in which $\theta$ is the adjoint map and $e$ stands for the localization
map. 
It follows from the diagram (3.3) that 
$\theta(\text{action}_\K)= \Theta \circ |\Delta\widetilde{\mu}|$. 
Moreover we have 
$\theta(\text{action})=e_\sharp \circ \lambda = \lambda_K \circ e$. 
Thus the commutativity of the diagram (3.4) implies that 
$e^*([e^\sharp \circ \Theta \circ |\Delta\widetilde{\mu}|]) = 
e^*([\lambda_\K])$ in $[G, \F(G/U, (G/K)_\K)]$. 
Since $G$ is connected, it follows that 
$(e^\sharp)\circ \Theta \circ |\Delta\widetilde{\mu}|\circ e \simeq 
\lambda_\K \circ e : G \to \F(G/U, (G/K)_\K; e\circ p)$.
The fact that $e_\sharp : \F(G/U, (G/K); p) \to\F(G/U, (G/K)_\K; e\circ p)$ 
is the localization yields that 
the induced map $e^* : [G_\K, \F(G/U, (G/K)_\K; e\circ p)] \to 
[G, \F(G/U, (G/K)_\K; e\circ p)]$ 
is bijective. 
This completes the proof. 
\end{proof}

Before proving Theorem \ref{thm:key}, we recall some maps. 
For a simplicial set $K$, 
there exists a natural homotopy equivalence 
$\xi_K : K \to \Delta|K|$, which is defined by 
$\xi_K(\sigma)=t_\sigma : \Delta^n \to \{\sigma\}\times \Delta \to |K|$. 
This gives rise to a  
quasi-isomorphism $\xi_A : \Omega \Delta A \stackrel{\simeq}{\to} 
\Omega \Delta |\Delta A|$. 
Moreover, we can define a bijection 
$
\eta : \text{DGA}(A, \Omega K) \stackrel{\cong}{\to} 
\text{Simp}(K, \Delta A)
$
by $\eta : \phi \mapsto f; f(\sigma)(a)=\phi(a)(\sigma)$, where $a \in
A$ and $\sigma \in K$. 
We observe that  
$\eta^{-1}(id) : A \to \Omega \Delta A$ is a quasi-isomorphism 
if $A$ is a connected Sullivan algebra; see \cite[10.1. Theorem]{B-G}.

\medskip
\noindent
{\it Proof of Theorem \ref{thm:key}.}
Let $\pi : \widetilde{E} \to \widetilde{E}/M_{\widetilde{u}}$ be the projection. 
With the same notation as above, 
we then have a commutative diagram 
$$
\newdir{ >}{{}*!/-10pt/@{>}}
\xymatrix@C12pt@R20pt{
 & |\Delta(\wedge(W\otimes B_*)/F)| 
\ar[rr]^{\Theta}_{\simeq} &  & \F((G/U)_\K,
 (G/K)_\K) \\ 
|\Delta(\wedge V_G)| \ar[r]_{|\Delta(\widetilde{\widetilde{\mu}})|}
  \ar[ur]^{|\Delta(\widetilde{\mu})|} & \hspace{0.1cm}
|\Delta(\widetilde{E}/M_{\widetilde{u}})| \ar@{ >->}_{|\Delta \pi|}[u] 
 \ar[r]_{|\Delta \pi|}^{\simeq} & |(\Delta \widetilde{E})_{\widetilde{u}}|   
\ar[r]_(0.25){\Theta}^(0.25){\simeq} &
\F((G/U)_\K,  (G/K)_\K; \Theta([(1, \widetilde{u})])),  \ar@{^{(}->}[u]
}
\eqnlabel{add-3}
$$ 
where $[(1, \widetilde{u})] \in |\Delta \widetilde{E}|$ is the element whose 
representative is $(1, \widetilde{u}) \in \Delta^0 \times (\Delta
\widetilde{E})_0$. 
Lemma \ref{lem:homotopy} yields that 
$$e^\sharp \circ \Theta \circ |\Delta \pi|\circ 
|\Delta\widetilde{\widetilde{\mu}}|
\simeq e^\sharp \circ \Theta \circ |\Delta \widetilde{\mu}| \simeq
\lambda_\K. 
\eqnlabel{add-4}
$$
Thus we see that 
$e^{\sharp}$ maps $\F((G/U)_\K,  (G/K)_\K; \Theta([(1, \widetilde{u})]))$ to 
$\F((G/U),  (G/K)_\K; e^{\sharp}\circ \Theta([(1, \widetilde{u})]))$, which is
the connected component containing $\text{Im}(\lambda_\K)$. 
This implies that 
$\F((G/U),  (G/K)_\K; e^{\sharp}\circ \Theta([(1, \widetilde{u})]))
=\F((G/U),  (G/K)_\K; e\circ p)$. 
Therefore, 
by the naturality of maps $\eta$ and $\xi_A$, 
we have a diagram
$$
\xymatrix@=18pt{
A_{PL}(G_\K) \ar@{=}[dd] 
      & A_{PL}(\F(G/U, (G/K)_\K; e\circ p)) \ar[l]_(0.6){A_{PL}(\lambda_\K)} 
                     \ar[d]^{((e^\sharp))^*} \\
   &   A_{PL}(\F((G/U)_\K, (G/K)_\K; \Theta([(1, \widetilde{u})]) )) 
\ar[d]^{\Theta^*} \\    
A_{PL}(|\Delta\wedge V_G|) &
A_{PL}(|\Delta(\widetilde{E}/M_{\widetilde{u}})|) 
=\Omega \Delta(\widetilde{E}/M_{\widetilde{u}})
  \ar[l]_(0.6){|\Delta \widetilde{\widetilde{\mu}}|^*} \\
\wedge V_G \ar[u]_{\simeq}^{t':=(\xi_{\wedge V_G}) \eta^{-1}(id)}& 
\widetilde{E}/M_{\widetilde{u}} 
\ar[u]^{\simeq}_{\xi_{\widetilde{E}/M_{\widetilde{u}}} \eta^{-1}(id)=:t} 
               \ar[l]_(0.6){\widetilde{\widetilde{\mu}}}  
}
\eqnlabel{add-4}
$$ 
in which the upper square is homotopy commutative and the lower square 
is strictly commutative. Lifting Lemma allows us to obtain a DGA map 
$\varphi : \widetilde{E}/M_{\widetilde{u}} 
\to A_{PL}(\F(G/U ,(G/K)_\K))$ such that  
$\Theta^*\circ ((e^\sharp)_*)^* \circ \varphi \simeq t$  
and hence $A_{PL}(\lambda_\K)\circ \varphi \simeq t' \circ
\widetilde{\mu}$. 

Given a space $X$, let $u : A \to A_{PL}(X)$ be a DGA map from a Sullivan
algebra $A$.  
Let $[f]$ be an element of $\pi_n(X)$ and $\iota : (\wedge Z, d) 
\stackrel{\simeq}{\to} A_{PL}(S^n)$ the minimal
model.  By taking a Sullivan
representative $\widetilde{f} : A \to \wedge Z$ with respect to $u$,
namely a DGA map satisfying the condition that 
$\iota \circ \widetilde{f}\simeq A_{PL}(f) \circ u$,    
we define a map $\nu_u : \pi_n(X) \to  \text{Hom}(H^nQ(A), \K)$
by $\nu_u ([f]) = H^nQ(\widetilde{f}) : H^nQ(A) \to H^nQ(\wedge Z)=
\K$. 
By virtue of \cite[6.4 Proposition]{B-G}, in particular, 
we have a commutative diagram 
$$
\xymatrix@C30pt@R15pt{
\pi_n(G_\K) \ar[r]^(0.4){\lambda_\K} \ar[d]_{\nu_{t'}}^{\cong} 
       & \pi_n(\F(G/U, (G/K)_\K); e \circ p ) \ar[d]^{\nu_\varphi}_{\cong} \\
\text{Hom}((V_G)^n, \K)  \ar[r]_(0.45){HQ(\widetilde{\widetilde{\mu}})^*} &
 \text{Hom}(H^nQ(\widetilde{E}/M_{\widetilde{u}}), \K).  
}
$$ 
in which $\nu_{t'}$and $\nu_\varphi$ are an isomorphism; 
see \cite[8.13 Proposition]{B-G}. 
There exists an element $[f_i]\otimes q$ in $\pi_*(G)\otimes \K$ which
corresponds to the dual element $x_i^*$ via the isomorphism 
$\pi_*(G)\otimes \K \cong \pi_*(G_\K) \stackrel{\nu_{t'}}{\to} 
\text{Hom}((V_G)^n, \K)$ for any $i = 1, ..., s$. 
The required map $\rho : \times_{j=1}^sS^{\deg x_i} \to G$ 
is defined by the composite of 
the map $\times_{j=1}^s f_i$ and the product $\times_{j=1}^sG \to G$. 
\hfill\qed 
\medskip

The existence of the commutative diagram (3.7) up to homotopy yields the
following result. 

\begin{thm}
\label{thm:model_adjoint}
The DGA map $\widetilde{\widetilde{\mu}} : \widetilde{E}/M_{\widetilde{u}} 
\to \wedge V_G$ 
is a model for the map $\lambda : G \to \F(G/U, G/K; p)$, namely a
 Sullivan representative in the sense of 
\cite[Definition, page 154]{F-H-T}. 
\end{thm}

\section{A model for the left translation}  
In order to prove Theorems \ref{thm:main2} and \ref{thm:main3}, a more
explicit model for the map $\lambda_{G,M} : G \to \aut{M}$ is required. 
To this end, we refine 
the model of the left translation described in the proof of 
Theorem \ref{thm:key}. 

We first observe that the cohomology $H^*(BU ; \K)$ is isomorphic to 
a polynomial algebra with finite generators, 
say $H^*(BU ; \K)\cong \K[h_1, ..., h_l]$. 
We consider a commutative diagram of fibrations 
$$
\xymatrix@C30pt@R15pt{
& G \ar@{=}[r] \ar[d]_i & G \ar[d] \\
G/U & G \times_U E_U \ar[l]^{\simeq}_h \ar[r] \ar[d]_{\overline{\pi}} 
& E_G \ar[d]^{\pi} \\
 & BU \ar[r]_{B\iota} & BG  
}
$$
in which $h : G \times_U E_U \to G/U$ is a homotopy equivalence defined
by $h([g, e])= [g]$.  
This diagram yields a Sullivan 
model $(\wedge W, d)$ for $G/U$ which has the form 
$
(\wedge W, d) = (\wedge (h_1, ..., h_l, x_1, ..., x_k), d)  
$
with $dx_j = (B\iota)^*c_j$; 
see \cite[Proposition 15.16]{F-H-T} for the details. 
Moreover we have a model  $(\wedge V_G , d)$ for $G$ of the form 
$(\wedge (x_1, ...., x_k), 0)$. Since $h\circ i$ is nothing but 
the projection $\pi : G \to  G/U$, it follows that 
the natural projection 
$\rho : (\wedge (h_1, ..., h_l, x_1, ..., x_k), d)  
\to (\wedge (x_1, ...., x_k), 0)$ 
is a Sullivan model for the map $\pi$.   

Let $\beta : G\times (G\times_UE_U) \to G \times_U E_G$ be the action of
$G$ on $G\times_UE_U$. Then the left translation $tr : G\times G/U \to G/U$ 
coincides with $\beta$ up to the homotopy equivalence 
$h : (G\times_UE_U)\to G/U$ mentioned above. 
Thus in order to obtain a model for  the linear action, 
it suffices to construct a model for $\beta$.    
Recall the fibration
$G \to G\times_U E_U \stackrel{\overline{\pi}}{\to} BU$ 
and the universal fibration 
$G \to E_G \stackrel{\pi}{\to} BG$. 
We here consider a commutative diagram 
$$
\xymatrix@C14pt@R10pt{
 & G\times (G\times_UE_U) \ar[dl]_{\beta} \ar@{->}'[d]_{\overline{\pi}'}[dd] 
 \ar[rr]^{1\times f} &   & G \times E_G \ar[dd]_{\pi'} \ar[dl]_{\alpha} \\
 G\times_UE_U \ar[rr]^(0.7){f} \ar[dd]_{\overline{\pi}} 
   & &  E_G \ar[dd]_(0.4)\pi \\
  &  BU \ar@{->}'[r]_{B\iota}[rr]  \ar[dl]_{=} &   & BG   \ar[dl]^{=} \\
 BU \ar[rr]_{B\iota} & & BG  \\
}
\eqnlabel{add-1}
$$
in which 
 $\pi'$ and $\overline{\pi}'$ are fibrations with 
the same fibre $G\times G$ 
and the restrictions $\alpha_{|\text{fibre}} : G\times G \to G$ and 
 $\beta_{|\text{fibre}} : G\times (G\times_UE_U) \to (G\times_UE_U)$ are 
the multiplication on $G$ and the  action of $G$, respectively. 
 Let 
$i : (\wedge V_{BG}, 0) \rightarrowtail \wedge (\widetilde{V_{BU}}, d)$ 
be a Sullivan model for $B\iota$. In particular, we can choose such a
model so that 
$$\wedge \widetilde{V_{BU}}= \wedge (c_1 ,...., c_m) 
\otimes \wedge (h_1, ...,h_l) \otimes \wedge (\tau_1, ..., \tau_m)
$$
and $d(\tau_i) = B\iota(c_i) - c_i$. 
By the construction of a model for pullback fibration mentioned 
in \cite[page 205]{F-H-T}, we obtain a diagram 
$$
\xymatrix@C14pt@R5pt{
 &  \wedge Z &   & \wedge W' \ar[ll]_(0.3){v'} \\
\wedge V \ar[ur]^{\widetilde{\beta}} &  
       & \wedge V' \ar[ll]_(0.3){v} \ar[ur]^{\widetilde{\alpha}} \\
  &  \wedge \widetilde{V_{BU}} \ar@{->}'[u][uu]_{u'} &   
                 & \hspace{0.2cm} 
 \wedge V_{BG}  \ar@{>->}'[l][ll]_(0.3)i \ar[uu] \\
\wedge \widetilde{V_{BU}} \ar[ur]^{=} \ar[uu]^{u}& 
                 & \hspace{0.2cm} \wedge V_{BG} \ar@{>->}[ll]^i
		  \ar[ur]_{=} \ar[uu] \\
}
\eqnlabel{add-2}
$$
in which vertical arrows are Sullivan models for the fibrations in the 
diagram (4.1). 
Observe that squares are commutative except for the top square.  
Let $\Psi : \wedge Z \to A_{PL}(G \times (G\times_UE_U))$ be the 
Sullivan model with which Sullivan representatives in (4.2) are
constructed. The argument in \cite[page 205]{F-H-T} allows us to choose 
homotopies, which makes maps $v$, $\widetilde{\beta}$, $v'$ and 
$\widetilde{\alpha}$ Sullivan
representatives for the corresponding maps, 
so that all of them are relative with respect to $\wedge V_{BG}$. 
This implies that 
$\Psi \circ \beta \circ v \simeq \Psi \circ v' \circ \widetilde{\alpha}$ 
rel $\wedge V_{BG}$. By virtue of Lifting lemma \cite[Proposition 14.6]{F-H-T},
we have a homotopy $H : \widetilde{\beta} 
\circ v \simeq v' \circ \widetilde{\alpha}$ rel $\wedge V_{BG}$. 
Thus we have a homotopy commutative diagram 
$$
\xymatrix@C20pt@R20pt{
\wedge V' \otimes_{\wedge V_{BG}}
\wedge \widetilde{V_{BU}} \ar[r]^(0.7){u\cdot v} 
  \ar[d]_{ \widetilde{\alpha} \otimes 1}  & \wedge V  
          \ar[d]^{\widetilde{\beta}} \\
\wedge W'  \otimes_{\wedge V_{BG}} \wedge \widetilde{V_{BU}}
\ar[r]_(0.7){u'\cdot v'} & \wedge Z
}
$$
in which horizontal arrows are quasi-isomorphisms; see 
\cite[(15.9) page 204]{F-H-T}. In fact the homotopy 
$K : \wedge \widetilde{V_{BU}}\otimes_{\wedge V_{BG}} \wedge V' 
\to \wedge W \otimes \wedge (t, dt)$ is given by 
$K = (\widetilde{\beta} \circ u)\cdot H$. 
Observe that $\widetilde{\beta}\circ u = u'$. 
Thus we have a model $\widetilde{\alpha}\otimes 1$ for $\widetilde{\beta}$ and
hence for the left translation. 

The model $\widetilde{\alpha}\otimes 1$ can be replaced 
by more tractable one. In fact, recalling the model 
$(\widetilde{V_{BU}}, d)$ for $BU$ mentioned above, 
it is readily seen that the map 
$s : \wedge \widetilde{V_{BU}} \to \wedge V_{BU}=\wedge (h_1, ...,
h_l)$, 
which is defined by $s(c_i) = (B\iota)^*(c_i)$, $s(h_i)=h_i$ 
and $s(\tau_j)=0$, is a quasi-isomorphism and 
is compatible with $\wedge V_{BG}$-action. 
Observe that the Sullivan representative for $B\iota : BU \to BG$ 
is also denoted by $(B\iota)^*$.    
Thus we have a commutative diagram 
$$
\xymatrix@C20pt@R20pt{
\wedge V'  \otimes_{\wedge V_{BG}} \wedge V_{BU}
\ar[d]_{{\zeta}:= \widetilde{\alpha} \otimes 1}  & 
 \wedge V' \otimes_{\wedge V_{BG}}\wedge \widetilde{V_{BU}}  
 \ar[d]^{\widetilde{\alpha} \otimes 1} \ar[l]_{1\otimes s} \\
\wedge W'  \otimes_{\wedge V_{BG}}\wedge V_{BU} & 
\wedge W'  \otimes_{\wedge V_{BG}}\wedge \widetilde{V_{BU}} 
\ar[l]^{1\otimes s} 
}
$$
in which the DGA maps $1\otimes s$ are quasi-isomorphisms.  
As usual, the Lifting lemma enables us to deduce the following lemma.  

\begin{lem}
\label{lem:model-action}
The DGA map $\zeta:= \widetilde{\alpha} \otimes 1$ 
is a Sullivan representative for 
the left translation $tr : G \times G/U \to G/U$. 
\end{lem}

In order to construct a model for $tr$ more explicitly, we proceed 
to construct that for $\alpha$. 

\begin{lem} 
\label{lem:model-action2}
There exists a Sullivan representative $\psi$ 
for $\alpha$ such that 
a diagram 
$$
\xymatrix@C15pt@R8pt{
 & \wedge (x_1, ..., x_l)\otimes \wedge V_{BG} \ar[dd]^{\psi} 
       & \hspace{-2.8cm} = \wedge V' \\
\wedge V_{BG} \hspace{0.2cm} \ar@{>->}[ur]^{i_1} \ar@{>->}[dr]^{i_2}& & \\
 & \wedge (x_1, ..., x_l)\otimes \wedge (x_1, ..., x_l)\otimes \wedge V_{BG}
   & \hspace{-0.6cm} = \wedge W' 
}
$$
is commutative and 
$
\psi(x_i) = x_i\otimes 1 \otimes 1 + 1\otimes x_i \otimes 1 
+ \sum_{n}X_n\otimes X_n' C_n 
$
for some monomials 
$X_n \in \wedge (x_1, ..., x_l)$, $X_n' \in \wedge^+(x_1, ..., x_l)$ 
and monomials $C_n \in \wedge^+V_{BG}$. Here  
$i_1$ and $i_2$ denote Sullivan models for $p$ and $p'$, respectively. 
\end{lem}

\begin{proof}
We first observe that $d(x_i \otimes 1) = 0$ and 
$d(1\otimes x_i) =c_i \in \wedge (c_1, ...,c_l)=\wedge V_{BG}$ in 
$\wedge W'$.  
It follows from \cite[15.9]{F-H-T} that there exists a DGA map $\psi$ 
which makes the diagram commutative. We write 
$$
\psi(x_i) = x_i\otimes 1 \otimes 1 
+ 1\otimes x_i \otimes 1 + \sum_{n}X_n\otimes X_n' C_n 
+ \sum_n \tilde{X}_n\otimes \tilde{X}_n' + \sum_n X_n''\otimes C_n''  
$$ 
with monomial bases,  
where 
$X_n,  X_n'' \in \wedge (x_1, ..., x_l)\otimes 1 \otimes 1$, 
$X_n' \in 1\otimes \wedge^+(x_1, ..., x_l)\otimes 1$,  
$\tilde{X}_n\otimes \tilde{X}_n' 
\in \wedge (x_1, ..., x_l)\otimes \wedge (x_1, ..., x_l)\otimes 1$ 
and $C_n, C_n'' \in \wedge^+V_{BG}$. 
The map $\wedge (x_1 ,..., x_l) \to 
\wedge (x_1 ,..., x_l)\otimes \wedge (x_1 ,..., x_l)$ 
induced by $\psi$ is a Sullivan 
representative for the product of $G$.  
This allows us to conclude  that  $\tilde {X}_n$ and $\tilde{X}_n'$ 
are in  $\wedge^+(x_1, ..., x_l)$.  
Since $\psi$ is a DGA map, it follows that 
$$dx_i = \psi(dx_i) 
     = dx_i + \sum_{n}X_n\otimes d(X_n') C_n 
+ \sum_n \tilde{X}_n\otimes d(\tilde{X}_n').    
$$
This implies that $\sum_{n}X_n\otimes d(X_n') C_n = 0$ and  
$\sum_n\tilde{X}_n\otimes d(\tilde{X}_n')=0$. 
Since the map $d : \wedge^+(x_1, ...,x_l) \to \wedge(x_1, ...,x_l)\otimes 
\wedge V_{BG}$ 
is a monomorphism, it follows that $\sum_n\tilde{X}_n\otimes \tilde{X}_n'= 0$. 
We write $C_n'' = c_{i_n}^{k_n}\tilde{C}_n$, where $k_n \geq 1$. 
Define a homotopy 
$$
H : \wedge (x_1, ..., x_l) \otimes \wedge V_{BG} 
\to \wedge (x_1, ..., x_l)\otimes \wedge (x_1, ..., x_l) 
\otimes \wedge V_{BG}\otimes \wedge (t, dt)
$$
by $H(c_i) = c_i\otimes 1$ and 
\begin{eqnarray*}
H(x_i) &=& x_i\otimes 1 \otimes 1 + 1\otimes x_i \otimes 1 
+ \sum_{n}X_n\otimes X_n' C_n \\
& & -  \sum_n X_n''\otimes x_{i_n} \otimes c_{i_n}^{k_n-1}\tilde{C}_n
\otimes dt 
+   \sum_n X_n''\otimes 1 \otimes c_{i_n}^{k_n}\tilde{C}_n\otimes t. 
\end{eqnarray*} 
Put $\widetilde{\psi} = (\e_0\otimes 1) \circ \psi$. 
We see that $\widetilde{\psi} \simeq \psi$ rel $\wedge V_{BG}$. 
This completes the proof. 
\end{proof}

\section{Proof of Theorem \ref{thm:main2}}

We prove Theorem \ref{thm:main2} by means of the model for the left
translation described in the previous section.  

\medskip
\noindent
{\it Proof of Theorem \ref{thm:main2}.} We adapt Theorem \ref{thm:key}. 
We recall the Sullivan model $(\wedge W, d)$ for $G/U$ mentioned in
Section 4. Observe that $(\wedge W, d)$ 
has the form 
$$
(\wedge W, d) = (\wedge (h_1, ..., h_l, x_1, ..., x_k), d)  
$$
with $dx_j = (B\iota)^*c_j$.

Let $l : (H^*(BU), 0) \to (\wedge W, d)$ be the inclusion and 
$$
\xymatrix@C15pt@R15pt{
k : (\wedge W, d) \ar[r] & 
 (\wedge (h_1, ..., h_l)/(dx_1, ..., dx_l),0) \ 
\ar@{>->}[r] &  (H^*(G/U), 0)
}
$$ 
the DGA map defined by 
$k(h_i)=(-1)^{\tau(|h_i|)}h_i$ and $k(x_i)=0$. 
Recall the DGA $\widetilde{E}=\wedge(\wedge W\otimes (\wedge W)_*)/I$
and the DGA map $\widetilde{\mu} : \widetilde{E} \to \wedge V_G$
mentioned in Section 3, where we use the model 
$\zeta : \wedge W \to \wedge V_G\otimes \wedge W$ for the action 
$G \times G/U \to G/U$ constructed in Lemmas \ref{lem:model-action} and 
 \ref{lem:model-action2} in order to define $\widetilde{\mu}$; see (3.1).  
Consider the composite 

\medskip
\noindent
$\theta : (H^*(BU)  \! : \! H^*(G/U))=
\wedge(H^*(BU)\otimes H_*(G/U))/I$  \\

$
\qquad \qquad \maprightu{l\otimes 1}  
\wedge(\wedge W\otimes H_*(G/U))/I  
\maprightu{1\otimes k^{\sharp}} 
\wedge(\wedge W\otimes (\wedge W)_*)/I=\widetilde{E}. 
$

\medskip
\noindent
Let $\widetilde{u} : \widetilde{E} \to \K$ be an augmentation defined by 
$\widetilde{u}=\varepsilon \circ\widetilde{\mu}$, 
where  $\varepsilon : \wedge V_G \to \K$ is the augmentation. Then we
have $\theta(M_u) \subset M_{\widetilde{u}}$.   In fact, since 
$i^*(h_i)=(-1)^{\tau(|h_i|)}k \circ l(h_i)$ and 
$\langle  h_i, k^{\sharp}b_* \rangle = 
\langle \zeta h_i, b_* \rangle$ 
for $h_i \in H^*(BU)$, it follows
that 
\begin{eqnarray*}
\theta(h_i\otimes b_* - u(h_i\otimes b_*)) &=& 
h_i\otimes k^{\sharp}b_* - \langle i^*h_i, b_* \rangle \\
&=& h_i\otimes k^{\sharp}b_* - 
(-1)^{\tau(|h|)}\langle k h_i, b_* \rangle \\
&=& h_i\otimes k^{\sharp}b_* - 
(-1)^{\tau(|h|)}\langle \zeta  h_i,  b_* \rangle \\
&=& h_i\otimes k^{\sharp}b_* - 
\widetilde{u}(h_i\otimes k^{\sharp}b_*). 
\end{eqnarray*}

Consider an element 
$
z:=x_{i_t}\otimes 1_* - (-1)^{\tau(|u_{t*}|)}x_{j_t}
\otimes k^{\sharp}(u_{t*}) \in Q(\widetilde{E}/M_{\widetilde{u}}). 
$
For any $\alpha \in \wedge W$,  
$\langle \alpha, d^{\sharp}k^{\sharp}u_{t*} \rangle= 
\langle kd \alpha, u_{t*} \rangle= 0. $ 
Therefore we see that, in  $Q(\widetilde{E}/M_{\widetilde{u}})$,  
$$\delta_0(z) =dx_{i_t}\otimes 1_* - (-1)^{\tau(|u_{t*}|)}dx_{j_t}
\otimes k^{\sharp}(u_{t*}) = 
\theta((B\iota)^*(c_{i_t})\otimes 1_* -(B\iota)^*(c_{j_t})\otimes
u_{t*})=0.
$$
The last equality follows from the assumption that 
$(B\iota)^*(c_{i_t})\otimes 1_* \equiv(B\iota)^*(c_{j_t})\otimes u_{t*}$
modulo decomposable elements in $(H^*(BU) : H^*(G/U))/M_u$. 
By using the notation in Lemma \ref{lem:model-action2}, we see that 
\begin{eqnarray*}
H^*Q(\widetilde{\widetilde{\mu}})(z) &=& \langle \zeta x_{i_t}, 1_*\rangle - 
\langle \zeta x_{j_t}, k^{\sharp}u_{t*} \rangle \\
&=& \langle x_{i_t}\otimes 1, 1_*\rangle - 
\langle \sum X_n\otimes X_n'C_n,  k^{\sharp}u_{t*} \rangle \\
&=& x_{i_t} - \sum X_n  \langle k(X_n')C_n,  u_{t*} \rangle \ = \ x_{i_t}.
\end{eqnarray*}
Observe that $k(X_n')=0$. By virtue of Theorem \ref{thm:key}, we have
the result.  
\hfill \qed

\medskip
\begin{rem}
\label{rem:anotherproof}
As for the latter half of Theorem \ref{thm:main2}, namely, in the case
 where 
$(B\iota)^*(c_{i_1})$, ..., $(B\iota)^*(c_{i_s})$ are decomposable, we
 have a very simple proof of the assertion. In fact, 
the composite of the evaluation map $ev_0 : \aut{G/U} \to G/U$ and 
the map $\lambda : G \to \aut{G/U}$ is nothing but 
the projection $\pi : G \to G/U$. 
We consider the model $\eta : (\wedge W, d)\to (\wedge V_G, 0)$ for
 $\pi$ mentioned in the proof of Theorem \ref{thm:main2}. 
Then we see that $HQ(\rho)(x_{i_t})=x_{i_t}$ for the map 
$HQ(\rho) : HQ(\wedge W) \to HQ(\wedge V_G)=V_G$.   
Observe that  $x_{i_t} \in HQ(\wedge W)$ since $(B\iota)^*(c_{i_t})$ is
 decomposable. The same argument as the proof of Theorem \ref{thm:key}
 enables us to conclude that there is a map 
$\rho : \times_{t=1}^sS^{\deg c_{i_t}-1} \to G$ such that 
$\pi_* \circ \rho_* : \pi_*( \times_{t=1}^sS^{\deg c_{i_t}-1}_\K) \to 
\pi_*(G_\K)$ is injective. Thus $\lambda_* \circ \rho_*$ is injective in 
 the rational homotopy.   
\end{rem}

\begin{rem}
\label{rem:c-suspension}
In the proof of Theorem \ref{thm:main2}, we construct a model for $G$ of
 the form $(\wedge (x_1, ...., x_k), 0)$. By virtue of  
 \cite[Proposition 15.13]{F-H-T}, we can choose the elements $x_j$ so that  
$\sigma^*(c_j)= x_j$, where 
$\xymatrix@C20pt{\sigma^* : H^*(BG) \ar[r]^{\pi^*}  &
 H^*(E_G, G) & H^*(G) \ar[l]_(0.4){\delta}^(0.4){\cong}
}$  denotes the cohomology suspension. 
\end{rem}


In the rest of this section, 
we describe a suitable model for $\F(G/U, (G/K)_\K; e\circ p)$ 
for proving Theorems \ref{thm:main3} and \ref{thm:main4}. 

Let $G$ be a connected Lie group, $U$ a connected maximal rank 
subgroup and $K$ another connected maximal rank subgroup which contains
$U$. We recall from Section 2 a Sullivan model
for the connected component $\F(G/U, (G/K)_\K; e\circ p)$ containing 
the composite $e\circ p$ of the function space $\F(G/U, (G/K)_\K)$,
where $e : G/K \to (G/K)_\K$ is the localization map.  

Let $\iota_1 : K \to G$ and $\iota_2 : U \to K$ be the inclusions and
put $\iota = \iota_1\circ\iota_2$. 
Let $\varphi_U : (\wedge W', d) \stackrel{\simeq}{\to} \Omega \Delta
(G/U)$ and 
$\varphi_K : (\wedge \widetilde{W}, d) \stackrel{\simeq}{\to} \Omega \Delta
(G/K)$ be the Sullivan models for $G/U$ and $G/K$, respectively, 
mentioned in the proof of Theorem \ref{thm:main2}; that is,  
$(\wedge W', d)=(\wedge (h_1, ..., h_l, x_1, ..., x_k), d)$
with $d(x_i)=(B\iota)^*(c_i)$ and 
$(\wedge \widetilde{W}, d)=(\wedge (e_1, ..., e_s, x_1, ..., x_k), d)$
with $d(x_i)=(B\iota_1)^*(c_i)$.  
By applying Lifting Lemma to the commutative diagram 
$$
\newdir{ >}{{}*!/-10pt/@{>}}
\xymatrix@C25pt@R20pt{ 
\wedge V_{BK} \ar[r]^{(B\iota_2)^*}   \ar@{ >->}[d] 
 & \wedge V_{BU} \ \ar@{>->}[r] & \wedge W' \ar[d]^{\varphi_U}\\
\wedge \widetilde{W} \ar[r]_(0.45){\varphi_K} & \Omega\Delta(G/K) 
\ar[r]_{\Omega\Delta(p)} & \Omega\Delta(G/U), 
}
$$
we have a diagram 
$$ 
\xymatrix@C30pt@R20pt{ 
H^*(G/U) & \wedge W' \ar[r]_(0.45){\simeq} \ar[l]_(0.4){k}^(0.4){\simeq}
   & \Omega \Delta (G/U) \\
H^*(G/K) \ar[u]^{p^*}& \wedge \widetilde{W} 
 \ar[r]^(0.45){\simeq} \ar[u]^{\varphi} 
  \ar[l]^(0.4){l}_(0.4){\simeq} & 
\Omega \Delta (G/K)  \ar[u]_{\Omega\Delta (p)}
}
\eqnlabel{add-3}
$$
in which the right square is homotopy commutative and the left that is
strictly commutative. In particular, $k(x_i)= 0$, $l(x_i)=0$ and 
$\varphi(e_i)= (B\iota_2)^*e_i$.

Let $w : \wedge W \to \wedge \widetilde{W}$ be a minimal model for 
$(\wedge\widetilde{W}, d)$ and 
$k^{\sharp} : (H^*(G/U))^\sharp \to (\wedge W')^\sharp$ the
dual to the map $k$.

As in Remark \ref{rem:variants}(ii), 
we construct the DGA $E'$ 
by using $(\wedge W', d)=(B, d_B)$ and $(\wedge W, d)$. Then we have a
sequence of quasi-isomorphisms 
$$
\xymatrix@C25pt@R10pt{ 
E' \ar[r]^(0.3){\gamma := 1\otimes k^{\sharp}}_(0.3){\simeq} & 
\wedge (\wedge W\otimes
(\wedge W')_*)/I \ar[r]^(0.45){w\otimes 1}_(0.45){\simeq} &
\wedge(\wedge\widetilde{W}\otimes (\wedge W')_*)/I = \widetilde{E}.  
}
$$
Moreover, we choose a model $\zeta'$ for the action 
$G\times G/U \stackrel{tr}{\to} G/U \stackrel{p}{\to} G/K$ defined by 
the composite $\zeta' : \wedge \widetilde{W} \stackrel{\zeta}{\to} 
\wedge V_G \otimes \wedge \widetilde{W} \stackrel{1\otimes \varphi}{\to}   
\wedge V_G \otimes \wedge W'$, where $\zeta$ is the Sullivan
representative for the left translation $tr$ 
mentioned in Lemmas \ref{lem:model-action} and \ref{lem:model-action2}. 
Then the map $\zeta'$ deduces a model 
$$
\widetilde{\widetilde{\mu}} : E'/M_u \to \wedge V_G
\eqnlabel{add-2}
$$ 
for 
$\lambda : G \to \F(G/U, (G/K)_\K ; e\circ p)$ as in Theorem
\ref{thm:key}. Observe that 
$$
\widetilde{\mu}(v_i\otimes e_j)=(-1)^{\tau(|e_j|)}
\langle (1\otimes \varphi)\zeta w(v_i), k^{\sharp}e_j \rangle 
\ \  \text{and} \ \ 
u=\varepsilon\circ \widetilde{\mu}, 
\eqnlabel{add-3}
$$
where $\varepsilon : \wedge V_G \to \K$ denotes the augmentation. 
In the next section, we shall prove Theorem  \ref{thm:main3} by using the
model $\widetilde{\widetilde{\mu}} : E'/M_u \to \wedge V_G$.

\section{Proof of Theorem \ref{thm:main3}}  

Let $G$ and $U$ be the Lie group $U(m+k)$ and a maximal rank
subgroup of the form $U(m_1)\times \cdots\times U(m_s) \times U(k)$,
respectively. Without loss of generality, we can assume that 
$m_1 \geq \cdots \geq m_s \geq k$. 
Let $K$ be the subgroup $U(m)\times U(k)$ of $U$, 
where $m=m_1 +\cdots +m_s$. 
Then the Leray-Serre spectral sequence with coefficients 
in the rational field for the fibration $p : G/U \to G/K$ with fibre
$K/U$ collapses at the $E_2$-term because 
the cohomologies of $G/K$ and of $K/U$ are algebras generated by
elements with even degree.  
Therefore it follows that the induced
map $p^* : H^*(G/K) \to H^*(G/U)$ is a monomorphism.     
In order to prove Theorem \ref{thm:main3}, 
we apply Theorem \ref{thm:key} to the function space 
$\F(G/U, G/K, p)$.

Let $P=\{S_1, ..., S_n\}$ be a family consisting of subsets of 
the finite ordered set $\{1, ..., s \}$ which satisfies the condition that 
$x < y$ whenever $x \in S_i$ and $y \in S_{i+1}$.  
Define $\sharp^l P$ to be the number of elements of the set 
$\{S_j\in P \ | \  |S_j| =l \}$. Let $k$ be a fixed integer.  
We call such the family $P$ a $(i_1, ..., i_k)$-type block partition 
of $\{1, ..., s \}$ if $\sharp^l P = i_l$ for $1 \leq l \leq k$. 
Let $Q^{(s)}_{i_1, ..., i_k}$ denote the number of 
$(i_1, ..., i_k)$-type block partitions of $\{1, ..., s \}$. 

We construct a minimal model explicitly for the homogeneous space 
$U(m+k)/U(m)\times U(k)$. Assume that $m\geq k$. 
As in the proof of Theorem \ref{thm:main2}, we have a Sullivan model 
for $U(m+k)/U(m)\times U(k)$ of the form 
$$
(\wedge \widetilde{W}, d)
=(\wedge (\tau_1, ...., \tau_{m+k}, c_1 ,...,c_k, c_1', ..., c_m' ), d)
$$ with 
$d\tau_l = \sum_{i+j=l}c_i'c_j$. 

\begin{lem}
\label{lem:mini-model}
There exists a sequence of quasi-isomorphisms
$$
\xymatrix@=20pt{ 
\wedge \widetilde{W} & \wedge W_{(1)} \ar[l]_{\simeq} & \cdots \ar[l]_{\simeq} &  
\wedge W_{(s)} \ar[l]_{\simeq} & \cdots \ar[l]_{\simeq} &  
\wedge W_{(m)} \ar[l]_{\simeq}
}
$$
in which, for any $s$, $(\wedge W_{(s)}, d_{(s)})$ is a DGA of the form 
$$
\wedge W_{(s)}= 
\wedge (\tau_{s+1}, ...., \tau_{m+k}, c_1 ,...,c_k, c_{s+1}', ..., c_m'
 ) 
 \ \ \  \text{with}
$$
\begin{eqnarray*}
d_{(s)}\tau_l &=& c_l' + c_{l-1}'c_1 + \cdots +c_{s+1}'c_{l-(s+1)}  \\
        & & + \sum_{i_1+2i_2+\cdots +ki_k=s}(-1)^{i_1+ \cdots +i_k}
    Q^{(s)}_{i_1, ..., i_k}c_1^{i_1}\cdots c_k^{i_k}c_{l-s} \\
     & & + \sum_{i_1+2i_2+\cdots +ki_k=s-1}(-1)^{i_1+ \cdots +i_k}
    Q^{(s-1)}_{i_1, ..., i_k}c_1^{i_1}\cdots c_k^{i_k}c_{l-(s-1)}  \\ 
    & & + \cdots + (-c_1)c_{l-1} + c_l   
\end{eqnarray*}
for $s+1 \leq l \leq m+k$, where $c_i = 0$ for $i < 0$ or $i > k$. 
\end{lem}   

\begin{proof}
We shall prove this lemma by induction on the integer $s$. 
We first observe that $d\tau_2= c_2' -c_1c_1 + c_2$ in $\wedge W_{(1)}$
because $Q^{(1)}_{i_1}=1$. 
Define a map 
$\varphi : \wedge W_{(1)} \to \wedge \widetilde{W}$ by 
$\varphi(c_i)=c_i$, $\varphi(c_j')=c_j'$ and 
$\varphi(\tau_2)=\tau_2-\tau_1c_1$. 
Since $d\tau_1=c_1' + c_1$ in $\wedge W$, it follows that 
$\varphi$ is a well-defined quasi-isomorphism. Suppose that 
$(\wedge W_{(s)}, d_{(s)})$ in the lemma can be constructed for some 
$s \leq m-1$. In particular, we have 
$$
d_{(s)}\tau_{s+1} = c_{s+1}' + \sum_{0\leq j \leq s}
\sum_{\ \ i_1+2i_2+\cdots + ki_k=j}(-1)^{i_1+ \cdots +i_k}
    Q^{(j)}_{i_1, ..., i_k}c_1^{i_1}\cdots + c_k^{i_k}c_{s+1-j}. 
$$  

\noindent
{\it Claim 1}. 
$$
Q^{(s+1)}_{i_1, ..., i_k} = Q^{(s)}_{i_1-1, i_2, ..., i_k}+ 
Q^{(s-1)}_{i_1, i_2-1,  ..., i_k} + \cdots + Q^{(s+1-k)}_{i_1, ...,
 i_k-1}. 
$$
Claim 1 implies that 
$$
d_{(s)}\tau_{s+1}= c_{s+1}' - \sum_{\ \ i_1+2i_2+\cdots + ki_k=s+1}
(-1)^{i_1+ \cdots + i_k}
    Q^{(s+1)}_{i_1, ..., i_k}c_1^{i_1}\cdots c_k^{i_k}. 
$$
We define $d_{(s+1)}\tau_{l+1}$ in $\wedge W_{(s+1)}$ by replacing the factor 
$c_{s+1}'$ which appears in $d_{(s)}\tau_{l+1}$ with 
$c_{s+1}' - d_{(s)}\tau_{s+1}$, namely, 
\begin{eqnarray*}
d_{(s+1)}\tau_{l+1} &=& c_{l+1}' 
+ c_l'c_1 + \cdots +c_{s+2}'c_{(l+1)-(s+2)}  \\
        & & + \sum_{i_1+2i_2+\cdots + ki_k=s+1}(-1)^{i_1+ \cdots + i_k}
    Q^{(s+1)}_{i_1, ..., i_k}c_1^{i_1}\cdots c_k^{i_k}c_{l-s} \\
    & & + \cdots + (-c_1)c_l + c_{l+1}.    
\end{eqnarray*}
Moreover define a map $\varphi : \wedge W_{(s+1)} \to \wedge W_{(s)}$ by 
$\varphi(c_i)=c_i$, $\varphi(c_j')=c_j'$ and 
$\varphi(\tau_{l+1})= \tau_{l+1} - \tau_{s+1}c_{l+1- (s+1)}$. 
It is readily seen that $\varphi$ is a well-defined DGA map. 
The usual spectral sequence argument enables us to deduce that 
$\varphi$ is a quasi-isomorphism. This finishes the proof. 
\end{proof}

\noindent
\medskip
{\it Proof of Claim 1.}
Let $\{P_l\}$ denote the family 
of all $(i_1, ..., i_k)$-type block partitions of $\{1, ..., s+1\}$. 
We write $P_l= \{S_1^{(l)}, ..., S_{n(l)}^{(l)}\}$.
Then $\{P_l\}$ is represented as the disjoint union of the families 
of $(i_1, ..., i_k)$-type block partitions whose last sets 
$S_{n(l)}^{(l)}$ consist of $j$ elements, namely,      
$\{P_l\} = \amalg_{1\leq j \leq k} \{P_l \ | \  |S_{n(l)}| = j \}$. 
It follows that 
$$\Bigl|\{P_l \ | \ |S_{n(l)}| = j \}\Bigr| = Q^{(s+1-j)}_{i_1, ..., i_{j-1}, i_j-1, 
i_{j+1}, ..., i_k}.
$$ 
We have the result. 
\hfill\qed

\medskip
Recall the minimal model $(\wedge W_{(m)}, d)$ for $G/K$ 
in Lemma \ref{lem:mini-model}. We see that 
$\deg d\tau_{m+1}= \deg c_1^{m}c_1 = 2(m+1)$ and that 
$d\alpha = 0$ for any element $\alpha$ with $\deg\alpha \leq 2m+1$. 
This yields that $c_1^m\neq 0$ in 
$H^*(G/K;\K)$. As mentioned before Lemma \ref{lem:mini-model}, the
induced map $p^* : H^*(G/K) \to H^*(G/U)$ is injective. Therefore we
have $(p^*c_1)^s\neq 0$ for $s \leq m$. 

Let $\widetilde{\widetilde{\mu}} : \widetilde{E}/M_u \to
\wedge V_G$ be the model for the map 
$\lambda : G \to \F(G/U, (G/K)_\K; e\circ p)$ 
mentioned in the previous section; see (5.2) and (5.3). 
The following four lemmas are keys to proving Theorem \ref{thm:main3}.
The proofs are deferred to the end of this section.  

\begin{lem}
\label{lem:delta-1}
$\delta_0(\tau_{m+(m-s+1)}\otimes ((p^*c_1)^m)_*) = (-1)^mc_{m-s+1}$ if 
$m\neq s$. 
\end{lem}

\begin{lem}
\label{lem:mu-1}
$\widetilde{\widetilde{\mu}}(\tau_{m+(m-s+1)}\otimes ((p^*c_1)^m)_*)= 0$ 
if $m\neq s$. 
\end{lem}

\begin{lem}
\label{lem:delta-2}
$\delta_0(\tau_{m+1}\otimes ((p^*c_1)^s)_*) = (-1)^ssc_{m-s+1}$. 
\end{lem}

\begin{lem}
\label{lem:mu-2}
$\widetilde{\widetilde{\mu}}(\tau_{m+1}\otimes ((p^*c_1)^s)_*)
= \tau_{m-s+1}$. 
\end{lem}

\medskip
\noindent
{\it Proof of Theorem \ref{thm:main3}}. 
By virtue of Lemmas \ref{lem:delta-1}, \ref{lem:mu-1}, 
\ref{lem:delta-2} and \ref{lem:mu-2}, we have 
\begin{eqnarray*}
&&\delta_0((-1)^m\tau_{m+(m-s+1)}\otimes ((p^*c_1)^m)_* 
 - \frac{(-1)^s}{s}\tau_{m+1}\otimes ((p^*c_1)^s)_*) \\ 
&& =  (-1)^m(-1)^mc_{m-s+1}-  \frac{(-1)^s}{s}(-1)^ssc_{m+s-1}= 0 
 \ \  \ \ \text{and} \\ 
&& \widetilde{\widetilde{\mu}}((-1)^m\tau_{m+(m-s+1)}\otimes ((p^*c_1)^m)_* 
 - \frac{(-1)^s}{s}\tau_{m+1}\otimes ((p^*c_1)^s)_*) \\
&& =  - \frac{(-1)^s}{s}\tau_{m-s+1}, 
\end{eqnarray*}
where $s \leq m-1$. Theorem \ref{thm:key} implies that 
$$
(\lambda_{\K})_i : \pi_i(G_\K) \to \pi_i(\F(G/U, (G/K)_\K, e\circ p))
$$
is injective for $i= \deg \tau_1, ...,\deg \tau_m$. 
Since $d\tau_l = \sum_{i+j}c_i'c_j$ in $(\wedge W)$, it follows that 
$d\tau_{l}$ is decomposable for $l \geq M+1$. Therefore  
Theorem \ref{thm:main2} yields that 
$(\lambda_{\K})_i$ is also injective for 
$i = \deg \tau_{m+1}, ..., \deg \tau_{m+k}$. 

The latter half of Theorem \ref{thm:main3} is obtained by comparing 
the dimension of rational homotopy groups. 
In fact, it follows from the rational model 
for $\aut{{\mathbb C}P^{m-1}}$ mentioned in Example
\ref{ex:evaluation-map} that 
\begin{eqnarray*}
\pi_*(\text{aut}_1({\mathbb C}P^{m-1})\otimes \K)^\sharp &\cong&
H_*(Q(\widetilde{E}/M_u), \delta_0) \\
&\cong& \K\{y\otimes 1_*, y\otimes(x^1)_*, ..., y\otimes (x^{m-2})_* \}.
\end{eqnarray*}
This implies that 
$\dim \pi_i(\text{aut}_1({\mathbb C}P^{m-1}))\otimes \K = 1 = 
\dim \pi_i(SU(m))\otimes \K$ for $i = 3, ..., 2m-1$. 
The result follows from the first assertion.  
This completes the proof.  
\hfill\qed

We conclude this section with proofs of Lemmas 
\ref{lem:delta-1}, \ref{lem:mu-1}, \ref{lem:delta-2} and \ref{lem:mu-2}.

\medskip
\noindent
{\it Proof of Lemma \ref{lem:delta-1}}. 
We regard the free algebra $\wedge (c_1, ..., c_l)$ 
as a primitively generated Hopf algebra. 
Observe that  $(c_i^s)_* = \frac{1}{s!}((c_i)_*)^s$.   
Recall the $0$-simplex $u$ in $\Delta E'$ mentioned in
(5.3). We have $u(c_j\otimes (p^*c_1)_*) = 0$ if $j \neq 1$ and  
\begin{eqnarray*}
u(c_1\otimes
 (p^*c_1)_*)&=&(-1)^{\tau(|p^*(c_1)|)}k^\sharp(p^*(c_1)_*)(\varphi\circ w(c_1)) \\
 &=&(-1)((p^*(c_1)_*)k \circ\varphi\circ w(c_1)=(-1)((p^*(c_1)_*)p^*c_1)= -1. 
\end{eqnarray*} 
For the map $k$ and $q$, see the diagram (5.1) and the ensuing paragraph. 
 Thus it follows that 
\begin{eqnarray*}
&& \delta_0(\tau_{m+(m-s+1)}\otimes ((p^*c_1)^m)_*) \\
&=& 
c_1^mc_{m-s+1}\cdot D^{(m)}(p^*c_1^m)_* = 
c_1^mc_{m-s+1}\cdot \frac{1}{m!}D^{(m)}(p^*c_1)_*^m \\
&=& \frac{1}{m!} c_1^mc_{m-s+1}\cdot \Bigl(
(p^*c_1)_*\otimes 1 \otimes \cdots \otimes
 1 + 1 \otimes (p^*c_1)_* \otimes 1 \otimes \cdots \otimes 1 \Bigr. \\ 
& & \qquad \qquad \qquad \qquad \qquad \qquad \qquad  \qquad \Bigl. + \cdots + 
1 \otimes \cdots \otimes 1 \otimes (p^*c_1)_* \Bigr)^m \\
&=& \frac{1}{m!} c_1^mc_{m-s+1}\cdot (\cdots + m! (p^*c_1)_* \otimes \cdots 
 \otimes (p^*c_1)_*\otimes 1 + \cdots ) \\
&=& u(c_1\otimes (p^*c_1)_*)\cdots u(c_1\otimes (p^*c_1)_*) c_{m-s+1} 
= (-1)^mc_{m-s+1}. 
\end{eqnarray*}
\hfill\qed

\medskip
\noindent
{\it Proof of Lemma \ref{lem:mu-1}}. Recall the quasi-isomorphism 
$\varphi_{s+1} : \wedge W_{(s+1)} \to \wedge W_{(s)}$ in the proof of  
Lemma \ref{lem:mini-model} 
which is defined by 
$\varphi(\tau_{l+1})=\tau_{s+1}-\tau_{l+1}c_{l+1-(s+1)}$. 
Let $w$ denote the composite 
$\varphi_{1}\circ \cdots \circ \varphi_{m} : \wedge W = \wedge W_{(m)}
\to \wedge \widetilde{W}$. It is readily seen that 
$w(\tau_{m+(m-s+1)})$ does not
have the element $c_1^m$ as a factor if $s\neq m$. 
Hence using the DGA map $\zeta'$ in Lemma \ref{lem:model-action},  
we have   
$$\widetilde{\widetilde{\mu}}(\tau_{m+(m-s+1)}\otimes ((p^*c_1)^m)_*)= 
(-1)^{\tau(|p^*c_1^m|)}\langle 
(1\otimes \varphi)\zeta w(\tau_{m+(m-s+1)}), k^\sharp(p^*c_1^m)_* 
\rangle = 0.$$
See (5.1) for the notations.  Observe that 
$H^*(G/K)\cong H^*(\wedge W)\cong \K[c_1, ..., c_k]$ for $* \leq 2m$. 
This completes the proof. \hfill\qed

\medskip
\noindent
{\it Proof of Lemma \ref{lem:delta-2}}. 
From Lemma \ref{lem:mini-model}, we see that in $\wedge W_{(m)}$,
\begin{eqnarray*}
d\tau_{m+1} &=& \sum_{i_1+2i_2+\cdots +ki_k=m}(-1)^{i_1+ \cdots +i_k}
    Q^{(m)}_{i_1, ..., i_k}c_1^{i_1}\cdots c_k^{i_k}c_{1} \\
        & & + \sum_{i_1+2i_2+\cdots +ki_k=m-1}(-1)^{i_1+ \cdots +i_k}
    Q^{(m-1)}_{i_1, ..., i_k}c_1^{i_1}\cdots c_k^{i_k}c_{2} \\
     & & + \cdots + \sum_{i_1+2i_2+\cdots +ki_k=l}(-1)^{i_1+ \cdots +i_k}
    Q^{(l)}_{i_1, ..., i_k}c_1^{i_1}\cdots c_k^{i_k}c_{m-l+1}  \\ 
    & & + \cdots  
\end{eqnarray*}
Suppose that $c_1^{i_1}\cdots c_k^{i_k}c_{m-l+1}\otimes ((p^*c_1)^s)_*
\neq 0$ in $Q(\widetilde{E}/M_u)$,  
where $i_1+2i_2+\cdots ki_k=l$. Then we have \\
(1) $l= m$ and $c_1^{i_1}\cdots c_k^{i_k}= c_1^{s-1}c_{m-s+1}$ or \\
(2)  $l\neq m$, $l=s$ and $c_1^{i_1}\cdots c_k^{i_k}= c_1^s$. \\
It follows that 
$(-1)^{i_1+ \cdots +i_k}Q^{(m)}_{i_1, ..., i_k}c_1^{s-1}c_{m-s+1}c_1
=(-1)^{s-1+1}(s-1)c_1^sc_{m-s+1}$ if 
$(i_1, ..., i_k) = (s-1, 0, ..., 0, 1, 0, ...,0 )$ with $i_{m-s+1}= 1$
and that $Q^{(s)}_{i_1, ..., i_k}c_1^sc_{m-s+1}
=(-1)^s\cdot 1 \cdot c_1^sc_{m-s+1}$ if $(i_1, ..., i_k) = (s, 0, ...,0 )$. 
This fact allows us to conclude that 
$\delta_0(\tau_{m+1}\otimes  ((p^*c_1)^s)_*)= (-1)^s(s-1)c_{m-s+1} + 
(-1)^sc_{m-s+1}=(-1)^ssc_{m-s+1}$. We have the result. 
\hfill\qed

\medskip
\noindent
{\it Proof of Lemma \ref{lem:mu-2}}. 
In order to compute $\widetilde{\widetilde{\mu}}$, 
we determine 
$\langle (1\otimes \varphi)\zeta w(\tau_{m+1}), k^\sharp(p^*c_1^s)_* \rangle$. 
With the the same notation as in the proof of Lemma \ref{lem:mu-1}, we
have $w(\tau_{m+1})= \cdots + (-1)^s\tau_{m-s+1}c_1^s + \cdots $. 
Lemmas \ref{lem:model-action} and \ref{lem:model-action2} imply that 
\begin{eqnarray*}
\zeta(\tau_{m-s+1}c_1^s)&=& \psi\otimes 1(\tau_{m-s+1}\otimes c_1^s) \\
&=&(\tau_{m-s+1}\otimes 1 \otimes 1 + 1\otimes \tau_{m-s+1} \otimes 1 
+ \sum_{n}X_n\otimes X_n' C_n)c_1^s.  
\end{eqnarray*}
Thus it follows that 
\begin{eqnarray*}
\widetilde{\widetilde{\mu}}(\tau_{m+1}\otimes ((p^*c_1)^s)_*)&=& 
(-1)^{\tau(|p^*c_1^s|)}\langle (1\otimes \varphi)\zeta w(\tau_{m+1}), 
k^\sharp(p^*c_1^s)_* 
\rangle \\
&=& (-1)^{s+s}\langle (1\otimes \varphi)\zeta(\tau_{m-s+1}c_1^s), 
k^\sharp(p^*c_1^s)_* 
\rangle \\
&=& \tau_{m-s+1}\langle \varphi(c_1^s), k^\sharp(p^*c_1^s)_* \rangle 
+ \langle \varphi(\tau_{m-s+1}c_1^s), k^\sharp(p^*c_1^s)_* \rangle  \\
& & + 
\sum_{n}X_n \langle \varphi(X_n'C_nc_1^s),k^\sharp(p^*c_1^s)_* \rangle \\
&=& \tau_{m-s+1}\langle k\varphi(c_1^s), (p^*c_1^s)_* \rangle 
+ \langle k\varphi(\tau_{m-s+1}c_1^s), (p^*c_1^s)_* \rangle  \\
& & + 
\sum_{n}X_n \langle k\varphi(X_n'C_nc_1^s), (p^*c_1^s)_* \rangle \\
&=& \tau_{m-s+1}. 
\end{eqnarray*}
The last equality is extracted from the commutativity of the diagram
(5.1). This completes the proof. 
\hfill\qed

\section{Proof of Theorem \ref{thm:main5}.}

This section is devoted to proving Theorem \ref{thm:main5}. 
The inclusion $\iota : \text{aut}_1(X) \to {\mathcal H}_{H, X}$ induces 
the map $B\iota : B\text{aut}_1(X) \to B{\mathcal H}_{H, X}$ with 
$B\iota \circ B\lambda_{G, X} = B\psi$. Therefore 
if $B\psi$ is injective on homology, then so is $B\lambda_{G, X}$. 

We shall prove the ``only if'' part by using the general categorical 
construction of a classifying space due to May \cite[Section 12]{May1} 
and by applying a part of the argument in the proof 
of \cite[Theorem 3.2]{May2} to our case.  

We here recall briefly the notion of a ${\mathcal O}$-graph; see 
\cite[page 68]{May2} for more detail. Let ${\mathcal O}$ be a discrete topological
space. 
Define a ${\mathcal O}$-graph to be a space ${\mathcal A}$ together
with maps $S : {\mathcal A} \to {\mathcal O}$ and 
$T : {\mathcal A} \to {\mathcal O}$.
The space ${\mathcal O}$ itself is regarded as ${\mathcal O}$-graph with
arrows $S$ and $T$ the identity map. 
Let ${\mathcal O}Gr$ be the category of ${\mathcal O}$-graphs whose morphisms
are maps $h : {\mathcal A} \to {\mathcal A}'$ compatible 
with maps $S$ and $T$. Observe that the pullback construction 
with respect to $S$ and $T$ makes  ${\mathcal O}Gr$ a monoidal
category. In fact, for  ${\mathcal O}$-graphs 
${\mathcal A}$ and ${\mathcal A}'$, ${\mathcal A}\Box{\mathcal A}'$ 
is defined by $\{(a, a') \in {\mathcal A}\times {\mathcal A}'| Sa =
Ta'\}$. 
Let ${\mathcal X}$ and  ${\mathcal Y}$ be a left ${\mathcal O}$-graph
and a right ${\mathcal O}$-graph, respectively; that is, ${\mathcal X}$
is a space with a map $T : {\mathcal X} \to {\mathcal O}$ and 
the space ${\mathcal Y}$ admits only a map $S : {\mathcal Y} \to
{\mathcal O}$. 

Let ${\mathcal M}$ be a monoid in 
${\mathcal O}Gr$ the category of ${\mathcal O}$-graphs and 
$B({\mathcal Y}, {\mathcal M}, {\mathcal X})$ denote the two-sided 
bar construction in the sense of May \cite[Section 12]{May1}, which is
the geometric realization of the simplicial space $B_*$ with 
$B_j= {\mathcal Y}\Box {\mathcal M}^{\Box j} \Box{\mathcal X}$.  
We regard a topological monoid $G$ as that in 
${\mathcal O}Gr$ with ${\mathcal O}=\{x\}$ the space of a point. 
Then the classifying space $BG$ we consider here 
is regarded as the bar construction $B(x, G, x)$.


\medskip
\noindent
{\it Proof of the ``only if'' part of Theorem  \ref{thm:main5}.}
Let $\iota' : {\mathcal H}_{H, X} \to \F(X, X)$ be the inclusion 
and  $e_* : \F(X, X)\to \F(X, X_\K)$ the map induced by 
the localization $e : X \to X_\K$. 
Since $X$ is an $F_0$-space or a space having the rational homotopy type
of the product of odd dimensional spheres by assumption, 
it follows from \cite[3.6 Corollary]{A-L} and 
\cite[Proposition 32.16]{F-H-T} that the natural map 
$
[X, X_\K] \to \text{Hom}(H^*(X_\K; \K), H^*(X; \K))
$
is bijective.  
We see that $e\circ \varphi \simeq e$ 
for any $\varphi \in {\mathcal H}_{H, X}$
Therefore the composite $e_*\circ \iota'$ factors through 
the connected component $\F(X, X_\K; e)$ of $\F(X, X_\K)$.  
We have a commutative diagram 
$$
\xymatrix@C20pt@R4pt{
{\mathcal H}_{H, X} \ar[rd]^(0.4){e_*\circ \iota'} & & \\
 & \F(X, X_\K; e) & \text{aut}_1(X_\K) \ar[l]_(0.5){e^*}^(0.5){\simeq} \\
\text{aut}_1(X)  \ar[ru]_(0.4){e_*} \ar[uu]^{\iota} & &
}
$$
in which the induced map $e^*$ is a homotopy equivalence.

Define ${\mathcal O}$  to be the discrete space with two points $x$ and
$y$.  Let ${\mathcal M}$ be the monoid in ${\mathcal O}Gr$ defined by 
${\mathcal M}(x, x)= \text{aut}_1(X)$, 
${\mathcal M}(y, y)= \text{aut}_1(X_\K)$
and ${\mathcal M}(x, y)= \F(X, X_\K ; e)$ with ${\mathcal M}(y, x)$ empty. 
Arrows $S, T : {\mathcal M}(a, b) \to {\mathcal O}$ are defined by 
$S(z)=a$ and $T(z)=b$ for $z \in {\mathcal M}(a, b)$.  
Moreover we define another monoid ${\mathcal M}'$ in ${\mathcal O}Gr$ by 
${\mathcal M}'(x, x)= {\mathcal H}_{H, X}$, 
${\mathcal M}'(y, y)= \text{aut}_1(X_\K)$, 
${\mathcal M}'(x, y)= \F(X, X_\K; e)$ and ${\mathcal M}'(y, x)=\phi$ 
with arrows defined immediately as mentioned above.   

The inclusions $i : \text{aut}_1(X) \to {\mathcal M}$, 
$j : \text{aut}_1(X_\K) \to {\mathcal M}$, 
$i': {\mathcal H}_{H, X}\to  {\mathcal M}'$ and 
$j' :  \text{aut}_1(X_\K) \to {\mathcal M}'$ induce the maps between
classifying spaces which fit into the commutative diagram 
$$
\xymatrix@C15pt@R5pt{
       &B{\mathcal H}_{H, X} \ar[r]^(0.45){Bi'} & 
 B({\mathcal O},{\mathcal M}',{\mathcal O} ) & \\
BG \ar[ru]^{B\psi} \ar[rd]_{B\lambda_{G,X}} & 
            & &   B\text{aut}_1(X_\K),  \ar[ul]^{\simeq}_{Bj'} 
 \ar[dl]_{\simeq}^{Bj}  \\
    & B\text{aut}_1(X)  \ar[r]_(0.45){Bi} \ar[uu]_{B\iota} &  
 B({\mathcal O},{\mathcal M},{\mathcal O} ) \ar[uu]_{B\widetilde{\iota}}& 
}
\eqnlabel{add-2}
$$
where  $\widetilde{\iota} : {\mathcal M} \to {\mathcal M}'$ is the
morphism of monoids in ${\mathcal O}Gr$ induced by the inclusion 
$\iota : \text{aut}_1(X) \to {\mathcal H}_{H, X}$.  
The proof of \cite[Theorem 3.2]{May2} enables us to conclude that 
maps $Bj$ and $Bj'$ are homotopy equivalences. 
The map $\Omega((Bj)^{-1}\circ (Bi))$ coincides with the composite 
$(e^*)^{-1}\circ e_* : \text{aut}_1(X) \to \F(X, X_\K; e) \to 
\text{aut}_1(X_\K)$ up to weak equivalence; 
see \cite[Theorem 3.2(i)]{May2}. 
Moreover the map $e_* : \text{aut}_1(X) \to \F(X, X_\K; e)$ is a
localization; see \cite{H-M-R}.  These facts yield that 
$\pi_*(\Omega Bi)\otimes \K$ is an isomorphism and hence so is 
$\pi_*(Bi)\otimes \K$. Thus the localized map $(Bi)_\K$ is a weak
equivalence.  
This implies that 
$(Bi)_* : H_*(B\text{aut}_1(X); \K) \to 
H_*(B({\mathcal O},{\mathcal M},{\mathcal O}); \K)$ is an isomorphism. 
The commutative diagram (7.1) enables us to conclude that 
$H_*(B\psi; \K)$ is injective if so is $H_*(B\lambda_{G,X} ; \K)$. 
This completes the proof. 
\hfill \qed 

\medskip

As we pointed out in the introduction,  \cite[Proposition 4.8]{K-M} follows from
Theorems \ref{thm:main3} and \ref{thm:main5}. In fact, suppose that $M$
is the flag manifold $U(m)/U(m_1)\times \cdots \times U(m_l)$ and 
$G=SU(m)$. Then as is seen in Remark \ref{rem:int} below 
$(\lambda_{G,M})_* : \pi_*(BG)\otimes \K \to \pi_*(B\aut{M})\otimes\K$
is injective if and only if 
$(B\lambda_{G,M})^* : H^*(BG) \to H^*(B\aut{M})$ is surjective.

\begin{rem}
 \label{rem:int}
Suppose that $M$ is a homogeneous space of the form $G/H$ for which 
$\text{rank} \ G = \text{rank} \ H$. 
The main theorem in \cite{S-T} due to Shiga and Tezuka implies 
that $\pi_{2i}(\text{aut}_1(M))\otimes \K= 0$ for any $i$. Thus 
$H^*(B\text{aut}_1(M); \K)$ is a polynomial algebra generated by 
the graded vector space $(sV)^\sharp$, 
where $(sV)_l = \pi_{l-1}(\text{aut}_1(M))$. Therefore the dual map 
to the Hurewicz homomorphism $\Xi^\sharp : H^*(B\text{aut}_1(M); \K)
 \to \text{Hom}(\pi_*(B\text{aut}_1(M)),  \K)$ induces 
an isomorphism on the vector space of indecomposable elements; 
see \cite[page 173]{F-H-T} for example. Thus
 the commutative diagram 
$$
\xymatrix@C28pt@R20pt{
H^*(BG; \K) \ar[d]_{\Xi^\sharp} 
    & H^*(B\text{aut}_1(M); \K) \ar[l]_{(B\lambda_{G,M})^*} 
         \ar[d]^{\Xi^\sharp} \\
\text{Hom}(\pi_*(BG), \K) & 
\text{Hom}(\pi_*(B\text{aut}_1(M)), \K) 
\ar[l]^(0.55){((B\lambda_{G,M})_*)^\sharp}
}
$$
yields that the map $(B\lambda_{G,M})^*$ is surjective 
if $G$ is rationally visible in $\aut{M}$. 
We also see that the induced map 
$(B\psi)_* : H_j(BG) \to H_j(B{\mathcal H}_{H, G/U})$ is injective
for each triple $(G, U, i)$ in Tables 1 and 2 if $j \in vd(G, G/U)$. 
\end{rem}

\section{The sets $vd(G, G/U)$ of visible degrees in Tables 1 and 2}

In this section, we deal with the visible degrees 
described in Tables 1 and 2 in Introduction. 

For the case where the 
homogeneous space $G/U$ has the rational homotopy type of the sphere, 
the assertions on the visible degrees follow from 
the latter half of Theorem \ref{thm:main2}.  In fact, the argument in
Example \ref{ex:spheres} does work well to obtain such results. 
The details are left to the reader.  The results for (11) and for (17)
follow from Theorems \ref{thm:main3} and  \ref{thm:main4},
respectively. 
We are left to verify the visible degrees 
for the cases (1), (5), (6), (6)' (16) and (19).  

(1). It is well-known that 
$(B\iota)^*(p_i)= (-1)^i(\chi^2 p_{i-1}' +  p_i')$ for the induced map 
$(B\iota)^* : H^*(BSO(2m+1)) \to H^*(B(SO(2)\times SO(2m-1))$, where 
 $p_i'$ is the $i$th Pontrjagin class in 
$H^*(B(SO(2m-1))\cong \K[p_1', ..., p_{m-1}']$; see  \cite{M-T}. 
 
We can construct a Sullivan model $(\wedge W, d)$ for 
the Grassmann manifold $M:=SO(2m+1)/SO(2)\times SO(2m-1)$ for which   
$\wedge W= \wedge (\chi, p_1', ..., p_{m-1}', \tau_2, \tau_4, ...,
 \tau_{2m})$
and $d(\tau_{2i}) = (-1)^i(\chi^2 p_{i-1}' +  p_i')$ for 
$1\leq i \leq m$. We see that there exists a quasi-isomorphism  
$w : (\wedge(\chi, \tau_{2m}), d\tau_{2m}=-\chi^{2m}) \to (\wedge W, d)$
 such that $w(\chi)=\chi$ and 
$$
w(\tau_{2m})=\chi^{2(m-1)}\tau_2 + \cdots + \chi^2\tau_{2(m-1)} +
 \tau_{2m}.   
$$
In view of the rational model 
$\widetilde{\widetilde{\mu}} : E'/M_u \to \wedge V_G$  
for $\lambda_{G,M} : SO(2m+1) \to \text{aut}_1(M)$ 
mentioned in (5.2) and Theorem \ref{thm:model_adjoint}, it follows
 from Lemma \ref{lem:model-action2} that 
\begin{eqnarray*}
\widetilde{\widetilde{\mu}}(\tau_{2m}\otimes (\chi^{2l})_*))&=&
 (-1)^{\tau(|\chi^{2l}|)}\langle \zeta \circ w(\tau_{2m}),
(\chi^{2l})_* \rangle \\
&=& \langle \chi^{2(m-1)}\tau_2 + \cdots + \chi^2\tau_{2(m-1)} + 
 \tau_{2m}, \ (\chi^{2l})_*\rangle \\ 
&=& \tau_{2(m-l)}, 
\end{eqnarray*}
where 
$\zeta$ is the Sullivan representative for the action 
$SO(2m-1) \times M \to M$ described in 
Lemma \ref{lem:model-action}. We have the result. 

The same argument does work well to prove the result for the case (16).

(19). Let $\iota : Spin(9) \to F_4$ be the inclusion map. 
Without loss of generality, we can assume that the induce map 
$$(B\iota)^* : H^*(BF_4; \K) = \K[y_4, y_{12}, y_{16}, y_{24}] \to 
H^*(BSpin(9); \K)= \K[y_4, y_8, y_{12}, y_{16}]$$
satisfies the condition that $(B\iota)^*(y_i)= y_i$ for $i = 4, 12, 16$ and 
 $(B\iota)^*(y_{24})= y_8^3$, where $\deg y_i = i$.  This fact follows from 
a usual argument with the Eilenberg-Moore spectral sequence for the
 fibration ${\mathcal L}P^2 \to BSpin(9) \stackrel{B\iota}{\to} BF_4$. 
By virtue of Lemmas \ref{lem:model-action} and
 \ref{lem:model-action2}, 
we see  that there exists a model for the linear action 
$F_4 \times {\mathcal L}P^2 \to {\mathcal L}P^2$ of the form 
$$
\zeta : ( \wedge (x_{23}')\otimes\wedge (y_8), d)
\to  (\wedge (x_3, x_{11}, x_{15}, x_{23})\otimes \wedge (x_{23}'
\otimes \wedge (y_8), d') 
$$
with 
$
\zeta(x_{23}')= x_{23}\otimes 1\otimes 1 + 1\otimes x_{23}' \otimes 1, 
$ 
where 
$d(x_{23}')= y_8^3$, $d'(x_j)= 0$ for $j = 3, 11, 15, 23$. 
In fact, for dimensional reasons, we write 
$\zeta(x_{23}')= 1\otimes x_{23}' \otimes 1 + 
 x_{23}\otimes 1 \otimes 1 + c  x_{15} \otimes 1\otimes y_8$ 
with a rational number $c$. 
By definition, we see that $\zeta= \psi\otimes 1$, where  
$\psi$ denotes the DGA map in Lemma \ref{lem:model-action2}.  
Since the image of each element with degree less than $24$ 
by $(B\iota)^*$ does not have the element 
$y_8$ as a factor, it follows that $c = 0$.  
Observe that $\wedge V_{BF_4}$-action on $\wedge V_{BSpin(9)}$ is
 induced by the map $(B\iota)^*$. 
The dual to the map $(\lambda_*)_i : \pi_i(F_4)\otimes \K \to
 \pi_i(\text{aut}_1(F_4/Spin(9)))\otimes \K$ is regarded as 
the induced map 
$$
H(Q(\widetilde{\widetilde{\mu}})) : 
H^* (Q(\widetilde{E}/M_u), \delta_0) \to 
V_G=\K\{x_3, x_{11}, x_{15}, x_{23} \} 
$$
in Theorem \ref{thm:key}.
We see that 
$
Q(\widetilde{E}/M_u)=\K\{y_8\otimes 1_*, x_{23}'\otimes 1_*,  
x_{23}'\otimes (y_8)_*, x_{23}'\otimes (y_8^2)_* \},  
$ 
$\delta_0( x_{23}'\otimes (y_8^2)_*)= 3 y_8\otimes 1_*$,  
$\delta_0( x_{23}'\otimes 1_*)=
\delta_0( x_{23}'\otimes (y_8^1)_*)=0$; 
see Example \ref{ex:evaluation-map}.  
Moreover the direct computation with (3.1) shows that 
$Q(\widetilde{\widetilde{\mu}})(x_{23}'\otimes 1_*)=\pm x_{23}$ and 
   $Q(\widetilde{\widetilde{\mu}})(x_{23}'\otimes (y_8)_*) = 0$. 
This implies that $vd(F_4, {\mathcal L}P^2)=\{23\}$. 


(5). The inclusion $\iota : SO(4) \to G_2$ induces the ring homomorphism 
$$
(B\iota)^* : H^*(BG_2)\cong \K[y_4, y_{12}]  \to H^*(BSO(4))\cong 
\K[p_1, \chi],  
$$ 
where $\deg p_1=4$ and $\deg \chi = 4$. It is immediate that 
$(B\iota)^*(y_{12})$ is decomposable for dimensional reasons. 
Form Example \ref{ex:evaluation-map}, we see that 
$\pi_*(\aut{{\mathbb H}P^2}) \cong \K\{y\otimes 1_*, y\otimes (x^1)_*\}$, 
where $\deg y\otimes 1_* = 11$ and $\deg y\otimes (x^1)_*=7$. It follows
from Theorem \ref{thm:main2} that $vd(G_2, G_2/SO(4))= \{11\}$.

(6). Let $T^2$ be the standard maximal torus of $U(2)$. 
We assume that $G_2 \supset U(2) \supset T^2$ without loss of
generality. Then the inclusion $W(G_2) \supset W(U(2))$ 
of Weyl groups gives the inclusions 
$$
\xymatrix@C20pt@R15pt{
\K[t_1, t_2]^{W(G_2)} \ \ar@{>->}[r] & 
\K[t_1, t_2]^{W(U(2))}  \ \ar@{>->}[r] & \K[t_1, t_2] \\
H^*(BG_2) \ar[u]^{\cong} & H^*(BU(2)) \ar[u]^{\cong}  
& H^*(BT^2). \ar[u]^{\cong}
}
$$
The result  \cite[page 212, Example 3]{L.Smith} implies that there exist
generators $y_4$, $y_{12}$ of $H(BG_2)$ such that 
$H(BG_2)\cong \K[y_4, y_{12}]$ and 
$y_4= t_1^2 -t_1t_2 + t_2^2$,  
$y_{12}=(t_1t_2^2 - t_1^2t_2)^2$ in $\K[t_1, t_2]^{W(G_2)}$. 
Since the Chern classes $c_1, c_2 \in H^*(BU(2))$ are regarded as 
$t_1 +t_2$ and $t_1t_2$, respectively in $\K[t_1, t_2]^{W(U(2))}$,  
it follows that 
$$
(B\iota)^*(y_4) = c_1^2 - 3c_2 \ \ \text{and} \ \  
(B\iota)^*(y_{12}) = c_1^2c_2^2 - 4c_2^3, 
$$ 
where $\iota : U(2) \to G_2$ is the inclusion. 
Put $\tilde{c}_2 = -\frac{1}{3}c_1^2 + c_2$. Then we see that 
$(B\iota)^*(-\frac{1}{3}y_4) = \tilde{c}_2$ and 
$$
(B\iota)^*(y_{12})= -\frac{1}{27}c_1^6 - \frac{2}{3}c_1^4\tilde{c}_2 
-3c_1^2\tilde{c}_2^2 -4\tilde{c}_2^3. 
$$
By the direct computation implies that 
\begin{eqnarray*}
& &(B\iota)^*(-\frac{1}{3}y_4)\otimes 1_* - 
(B\iota)^*(y_{12})\otimes (-\frac{3}{2})(c_1^4)_* \\ 
&=&   \tilde{c}_2\otimes 1_* + 
\frac{3}{2} \bigl(-\frac{1}{27}c_1^6 - \frac{2}{3}c_1^4\tilde{c}_2 
-3c_1^2\tilde{c}_2^2 -4\tilde{c}_2^3\bigr) \otimes (c_1^4)_* \\
&\equiv&  \tilde{c}_2\otimes 1_* -  \tilde{c}_2\otimes 1_* \equiv 0
\end{eqnarray*}
modulo decomposable elements in 
$(H^*(BU(2)) \! : \! H^*(G_2/U(2)))/M_u$.  
It is immediate that $(B\iota)^*(y_{12})$ is decomposable. 
By virtue of Theorem \ref{thm:main2}, we have 
$vd(G_2, G_2/U(2))=\{3, 11\}$. The same argument works well to obtain the
result for the case (6)'.

\medskip
\noindent
{\it Acknowledgments.}
I am particularly grateful to Kojin Abe, Kohhei Yamaguchi and Masaki
Kameko for valuable comments on this work and to Hiroo Shiga for drawing
my attention to this subject. I thank the anonymous referee of 
a previous version of this paper for showing me  
a very simple proof of the latter half of Theorem \ref{thm:main2}, which
is described in Remark \ref{rem:anotherproof}.

\section{Appendix. Extensions of characteristic classes}

For a space $X$, let $X^\delta$ denote the space with the 
discrete topology whose underlying set is the same as that of $X$. 
Let $M$ be a homogeneous space  admitting an action of 
a connected Lie group $G$. In this section, 
we consider cohomology classes of $B(\text{Diff}_1(M))^\delta$ as well
as those of $B(\aut{M})^\delta$, which detect familiar characteristic classes 
via the induced map 
$$
(B\lambda)^* : H^*(B(\text{Diff}_1(M))^\delta ; \K) \to H^*(BG^\delta; \K).    
$$

Let $G$ be a real semi-simple connected Lie group with finitely many
components and 
$h : G \to G_{\mathbb C}$ the complexification of $G$. 
One has a commutative diagram 
$$
\xymatrix@C20pt@R15pt{
H^*(BG_{\mathbb C}) \ar[r]^{h^*} & H^*(BG) \ar[r]^{j^*} & H^*(B(G^\delta)) \\ 
H^*(B\text{aut}_1(G/U)) \ar[ur]_{B\lambda^*} \ar[r] & 
H^*(B\text{Diff}_1(G/U)) \ar[u] \ar[r] &  
H^*(B(\text{Diff}_1(G/U))^\delta ) \ , \ar[u]  
}
$$
where $j : G^\delta \to G$ stands for the natural map.  
The result \cite[THEOREM 2]{M} asserts that the kernel of $j^*$
is equal to the ideal generated by the positive dimensional elements 
in $\text{Im} h^*$. 

As an example, we consider the case where $G = SL(2m; {\mathbb R})$ and 
$U$ is a maximal rank subgroup of $SO(2m)$ with $(QH^*(BU ; \K))^{2m}=0$, for
example $U$ is a maximal torus of $SO(2m)$.  
Observe that $G \cong SO(2m)$ and $G_{\mathbb C}\cong SU(2m)$. 
Then Milnor's result mentioned above allows us to conclude that 
the Euler class $\chi$ of $H^*(BSL(2m; {\mathbb R}))$ survives in 
$H^*(B(G^\delta))$; see \cite{Milnor-S}. Moreover 
Theorem \ref{thm:main2} yields that 
$(B\lambda)^* : H^i(B\aut{G/U}) \to
H^i(BG)$ 
is surjective for $i = 2m$; see also Remark \ref{rem:int}. 
Thus the class $\chi \in H^*(B(G)^\delta)$ 
is extendable to an element $\widetilde{\chi}$ of 
$H^*(B(\text{Diff}_1(G/U))^\delta)$. Let $p_m$ be the $m$th Pontjagin
class. Then we see that $h^*(c_{2m})=p_m$ and $p_m=\chi^2$; see \cite{M-T}.
This yields that $\chi^2=0$ in $H^*(B(G^\delta))$. 
It remains to show whether $\widetilde{\chi}^2$ is zero in 
$H^*(B(\text{Diff}_1(G/U))^\delta )$.

\begin{rem}
The result \cite[Corollary]{M} yields that the induced homomorphism 
$(Bj)^* : H^*(BG; {\mathbb Z}) \to H^*(BG^\delta; {\mathbb Z})$ is
 injective. Thus the same argument as above does work well to find
 a nontrivial element in the cohomology 
$H^*(B(\text{Diff}_1(G/U))^\delta)$ for an appropriate subgroup 
$U$ of $G$ if $H^*(BG; {\mathbb Z})$ is
 torsion free.   
\end{rem}


\begin{thebibliography}{99}
\bibitem{A-B-K}P. L. Antonelli, D. Burghelea and P. J. Kahl, The
	non-finite homotopy type of some diffeomorphism groups,
	Topology {\bf 11}(1972), 1-49.  
%
\bibitem{A-L}M. Arkowitz and G. Lupton, 
On finiteness of subgroups of self-homotopy equivalences, 
Contemp. Math. {\bf 181}(1995), 1-25. 
%
\bibitem{B-G}A. K. Bousfield  and V. K. A. M. 
Gugenheim, On PL de Rham theory and rational homotopy type, 
Memoirs of AMS {\bf 179}(1976).
%
\bibitem{B-L}J. Block and A. Lazarev, Andr\'e-Quillen cohomology 
and rational homotopy of function spaces.  Adv. Math. {\bf 193}(2005), 18-39.
%
\bibitem{B-S}E. H. Brown Jr and R. H. Szczarba, On the rational homotopy 
type of function spaces, Trans. Amer. Math. Soc. {\bf 349}(1997), 4931-4951.
%
\bibitem{B-F-M}U. Buijs, Y. F\'elix and A. Murillo, 
Lie models for the components of sections of a nilpotent fibration, 
Trans. Amer. Math. Soc.  {\bf 361}(2009), 5601-5614.
%
\bibitem{B-M}U. Buijs and A. Murillo, Basic constructions in rational
	homotopy theory of function spaces, 
Ann. Inst. Fourier (Grenoble) {\bf 56}(2006), 815-838. 
%
\bibitem{B-M2}U. Buijs and A. Murillo, 
The rational homotopy Lie algebra of function spaces, 
Comment. Math. Helv.  {\bf 83}(2008), 723-739.
%
\bibitem{DGMS}P. Deligne, P. Griffiths, J. Morgan and D. Sullivan, Real
	homotopy theory of K\"ahler manifolds, Invent. Math. 
         {\bf 29}(1975), 245-274. 
%
\bibitem{F-H}F. T. Farrell and W. C. Hsiang, 
On the rational homotopy groups of the diffeomorhism groups of discs, 
spheres and aspherical manifolds, 
Proceedings of Symposia in Pure Math. {\bf 32}(1978), 325-337. 
%
\bibitem{F-H-T}Y. F\'elix, S. Halperin and  J. -C. Thomas, 
Rational Homotopy Theory,
Graduate Texts in Mathematics {\bf 205}, Springer-Verlag.
%
\bibitem{F-T}Y. F\'elix and J. -C. Thomas, 
The monoid of self-homotopy equivalences of some homogeneous spaces,
Exposiotiones Math. {\bf 12}, 305-322.
%
\bibitem{G}D. H. Gottlieb, On fibre spaces and the evaluation map,
Ann. of Math. (2){\bf 87}(1968), 42-55.
%
\bibitem{G-M}V. K. A. M. Gugenheim and J. P. May, On the theory and
	applications of differential torsion products, 
Mem. Amer. Math. Soc. 142, 1974. 
%
\bibitem{H}A. Haefliger, Rational homotopy of space of sections of 
a nilpotent bundle, Trans. Amer. Math. Soc. {\bf 273}(1982), 609-620.
%
\bibitem{H-M-R}P. Hilton, G. Mislin and J. Roitberg, Localization of
nilpotent groups and spaces, North Holland Mathematics Studies
15, North Holland, New York, 1975.
%
\bibitem{H-K-O}Y. Hirato, K. Kuribayashi and N. Oda, 
A function space model approach to the rational evaluation subgroups,
 Math. Z. {\bf 258}(2008), 521-555.   
%
%
%
\bibitem{K-M}J. Kedra and D. McDuff, Homotopy properties of Hamiltonian
	group actions, Geometry \& Topology, {\bf 9}(2005), 121-162.  
%
\bibitem{Ku1}K. Kuribayashi, A rational model for the evaluation map, 
Georgian Mathematical Journal {\bf 13}(2006), 127-141. 
%
%
%
%
\bibitem{L-O}G. Lupton and J. Oprea, Cohomologically symplectic space:
Toral actions and the Gottlieb group, 
Trans. Amer. Math. Soc. {\bf 347}(1995), 261-288.
%
\bibitem{L-S}G. Lupton and S. B. Smith, 
Rationalized evaluation subgroups of a map. I. 
Sullivan models, derivations and $G$-sequences, 
J. Pure Appl. Algebra {\bf 209}(2007), 159-171. 
%
\bibitem{L-S2}G. Lupton and S. B. Smith, 
Rationalized evaluation subgroups of a map. II. 
Quillen models and adjoint maps,   
J. Pure Appl. Algebra  {\bf 209}(2007), 173-188. 
%
\bibitem{L-S3}G. Lupton and S. B. Smith, 
Whitehead products in function spaces: 
Quillen model formulae,  J. Math. Soc. Japan {\bf 62}(2010), 49-81.
%
\bibitem{May1}J.P.May, Classifying spaces and fibrations, 
Mem. Amer. Math. Soc. 155, 1975. 
%
\bibitem{May2}J.P.May, Fiberwise localization and completion, 
Trans. Amer. Math. Soc. {\bf 258}(1980), 127-146.
%
\bibitem{M-S}D. McDuff and D. Salamon, Introduction to Symplectic
	Topology, Oxford Mathematical Monographs, Clarendon Press,
	Oxford, 1995.
%
\bibitem{M}J. Milnor, On the homology of Lie groups made discrete, 
Comment. Math. Helvetici {\bf 58}(1983), 72-85. 
%
\bibitem{M-M}J. Milnor and J. -C. Moore, On the structure of Hopf algebras.
Ann. of Math. {\bf 81} (1965), 211-264.  
%
\bibitem{Milnor-S}J. Milnor and J. D. Stasheff, 
Characteristic classes, Annals of Mathematics Studies, No. 76. 
Princeton University Press, Princeton, 1974.
%
\bibitem{M-T}M. Mimura and H. Toda, Topology of Lie groups. I, II,  
Translations of Mathematical Monographs, 91. American Mathematical
	Society, Providence, RI, 1991.
%
\bibitem{N-S1}D. Notbohm and L. Smith, Fake Lie groups and maximal
	tori. III, Math. Ann {\bf 290}(1991), 629-642. 
%
\bibitem{N-S}D. Notbohm and L. Smith, Rational Homotopy of the space of
	homotopy equivalences of a flag manifold, 
Algebraic topology, Homotopy and Group Cohomology, Proceedings barcelona
	1990, J. Aguad\'e, M. Castellt and F. R. Cohen editors, 
Lecture Notes in Math. 1509 (1992), 301-312.  
%
\bibitem{O}A. L. Oniscik, Transitive compact transformation groups,
	Mat. Sb. {\bf 60}(1963), 447-485 [Russian],
	Amer. Math. Soc. Transl. {\bf 55}(1966), 153-194.  
%
%
%
\bibitem{S}S. Sasao, The homotopy of $\text{MAP}({\mathbb C}P^m, 
{\mathbb C}P^n)$, J. London Math. Soc. {\bf 8}(1974), 193-197.
%
\bibitem{S-T}H. Shiga and M. Tezuka, 
Rational fibrations, homogeneous spaces with positive Euler
	characteristics and Jacobians.  Ann. Inst. Fourier (Grenoble)
	{\bf 37}(1987), 81-106. 
%
\bibitem{L.Smith}L. Smith, Polynomial invariants of finite groups. 
Research Notes in Math., 6. A K Peters, Ltd., Wellesley, MA, 1995.
%
\bibitem{S.Smith}S. B. Smith, Rational evaluation subgroups,
Math. Z. {\bf 221}(1996), 387-400.
%
\bibitem{S.Smith2}S. B. Smith, 
Rational classification of simple function space components 
for flag manifolds, Canad. J. Math.  {\bf 49}(1997), 855--864. 
%
%
%
%
%
\bibitem{Y}K. Yamaguchi, On rational homotopy of $\text{MAP}({\mathbb H}P^m, 
{\mathbb H}P^n)$, Kodai Math. J. {\bf 6}(1983), 279-288.  
\end{thebibliography}
\end{document}